\documentclass[12pt,reqno ]{amsart}

\calclayout

\usepackage{amssymb}
\usepackage{amsmath}
\usepackage{amsrefs}
\usepackage{amsthm}
\usepackage{mathrsfs}
\usepackage{comment}
\usepackage{enumerate,color}

\makeatletter
\newsavebox{\@brx}
\newcommand{\llangle}[1][]{\savebox{\@brx}{\(\m@th{#1\langle}\)}%
  \mathopen{\copy\@brx\mkern2mu\kern-0.9\wd\@brx\usebox{\@brx}}}
\newcommand{\rrangle}[1][]{\savebox{\@brx}{\(\m@th{#1\rangle}\)}%
  \mathclose{\copy\@brx\mkern2mu\kern-0.9\wd\@brx\usebox{\@brx}}}
\makeatother

\usepackage[colorlinks=true, pdfstartview=FitV, linkcolor=blue,
citecolor=blue, urlcolor=blue]{hyperref}

\numberwithin{equation}{section}

\newtheorem{theorem}{Theorem}[section]
\newtheorem{lemma}[theorem]{Lemma}
\newtheorem{corollary}[theorem]{Corollary}

\newtheorem{proposition}[theorem]{Proposition}

\theoremstyle{remark}
\newtheorem{remark}[theorem]{\textbf{Remark}}

\def\Xint#1{\mathchoice
{\XXint\displaystyle\textstyle{#1}}%
{\XXint\textstyle\scriptstyle{#1}}%
{\XXint\scriptstyle\scriptscriptstyle{#1}}%
{\XXint\scriptscriptstyle\scriptscriptstyle{#1}}%
\!\int}
\def\XXint#1#2#3{{\setbox0=\hbox{$#1{#2#3}{\int}$}
\vcenter{\hbox{$#2#3$}}\kern-.5\wd0}}

\def\dashint{\Xint-}
\def\avgint{\Xint-}

\newcommand{\lla}{\llangle}
\newcommand{\rra}{\rrangle}

\newcommand{\loc}{\ensuremath{\text{loc}}}
\newcommand{\C}{\ensuremath{\mathbb{C}}}
\newcommand{\Cn}{\ensuremath{\mathbb{C}^n}}

\newcommand{\R}{\ensuremath{\mathbb{R}}}
\newcommand{\Z}{\ensuremath{\mathbb{Z}}}
\newcommand{\Rd}{\ensuremath{\mathbb{R}^d}}
\newcommand{\D}{\ensuremath{\mathscr{D}}}

\newcommand{\MC}[1]{\ensuremath{\mathcal{#1}}}

\newcommand{\supp}{\text{supp }}

\newcommand{\ip}[2]{\ensuremath{\left\langle#1,#2\right\rangle}}

\newcommand{\W}[1]{\ensuremath{\widetilde{#1}}}

\newcommand{\A}{\ensuremath{{\mathcal A}}}
\newcommand{\B}{\ensuremath{{\mathcal B}}}
\newcommand{\al}{\alpha}
\newcommand{\ra}{\rightarrow}

\newcommand{\Q}{\ensuremath{\mathcal{Q}}}
\newcommand{\OP}[1]{\ensuremath{\left|#1\right|_{\text{op}}}}

\DeclareMathOperator{\dv}{div}

\newcommand{\op}{\mathrm{op}}
\newcommand{\sca}{{\mathrm{sc}}}

\DeclareMathOperator*{\esssup}{ess\,sup}
\newcommand{\Uu}{\mathcal{U}}
\newcommand{\Vv}{\mathcal{V}}
\newcommand{\Dd}{\mathcal{D}}
\newcommand{\Ss}{\mathcal{S}}

\allowdisplaybreaks

\title{Two weight bump conditions for matrix weights}

\author{David Cruz-Uribe, OFS}
\address{Department of Mathematics\\
University of Alabama, Box 870350, Tuscaloosa, AL 35487.}\email{dcruzuribe@ua.edu}

\author{Joshua Isralowitz}
\address{Department of Mathematics and Statistics\\
SUNY Albany, 1400 Washington Ave., Albany, NY 12222.}\email{jisralowitz@albany.edu}

\author{Kabe Moen}
\address{Department of Mathematics\\
University of Alabama, Box 870350, Tuscaloosa, AL 35487.}\email{kabe.moen@ua.edu}

\begin{document}



\thanks{ The first
  author is supported by NSF Grant DMS-1362425 and research funds from the
  Dean of the College of Arts \& Sciences, the University of Alabama.
The second and third authors are supported by the Simons Foundation.}

\begin{abstract}
In this paper we extend the theory of two weight, $A_p$ bump conditions to the
setting of matrix weights.  We prove two matrix weight inequalities
for fractional maximal operators, fractional and singular integrals,
sparse operators and averaging operators.   As applications we prove
quantitative, one weight estimates, in terms of the matrix $A_p$ constant, for
singular integrals, and prove a Poincar\'e inequality related to those
that appear in the study of degenerate elliptic PDEs.
\end{abstract}

\maketitle

\section{Introduction}

In this paper we extend the theory of $A_p$ bump conditions to matrix
weights.  To put our results into context we first briefly review the
theory in the case of scalar weights.
A scalar weight $w$ (i.e., a non-negative, locally integrable
function) satisfies the Muckenhoupt $A_p$ condition, $1<p<\infty$, if
\[ [w]_{A_p} = \sup_Q \avgint_Q w\,dx \left(\avgint_Q
    w^{1-p'}\,dx\right)^{p-1} < \infty, \]
where here and below the supremum is taken over all cubes $Q$ with
edges parallel to the coordinate axes.    It is well known that this
condition is sufficient for a wide variety of classical operators
(e.g., the Hardy-Littlewood maximal operator, singular integral
operators) to be bounded on $L^p(w)$.    (Cf.~\cite{duoandikoetxea01,grafakos08b}.)

This condition naturally extends to pairs of weights:  we say
$(u,v)\in A_p$ if
\[ [u,v]_{A_p} = \sup_Q \avgint_Q u\,dx \left(\avgint_Q
    v^{1-p'}\,dx\right)^{p-1} < \infty. \]
However, unlike in the one weight case, while this condition is often
necessary for an operator to map $L^p(v)$ into $L^p(u)$, it is almost never
sufficient.  (See~\cite{MR2797562} and the references it contains.)   Therefore, for many years, the problem was to find a
similar condition that was sufficient.   The idea of $A_p$ bump
conditions originated with Neugebauer~\cite{neugebauer83} but  was
fully developed by  P\'erez~\cite{P2,perez94b,perez94}.  (See also
Sawyer and Wheeden~\cite{SW}.)  If we rewrite
the two weight $A_p$ condition as
\[ \sup_Q |Q|^{-1}\|u^{\frac{1}{p}}\|_{p,Q}\|v^{-\frac{1}{p}}\|_{p',Q} < \infty, \]
where $\|\cdot \|_{p,Q}$ denotes the localized $L^p$ norm with respect
to measure $|Q|^{-1}\chi_Q\,dx$, then a ``bumped'' $A_p$ condition is
gotten by replacing the $L^p$ and/or $L^{p'}$ norms with a slightly
larger norm in the scale of Orlicz spaces.

We recall a few properties of Orlicz spaces; for more details
see~\cite{MR2797562}.  Let $\Phi : [0,\infty)\rightarrow [0,\infty)$
be a Young function: convex, increasing, $\Phi(0)=0$, and
$\Phi(t)/t\ra \infty$ as $t\ra \infty$.  Given $\Phi$, its
associate function is another Young function defined by
\[ \bar{\Phi}(t) = \sup_{s>0} \big\{ st - \Phi(s)\big\}. \]
If $\Phi(t)=t^p$, $\bar{\Phi}(t)\approx t^{p'}$.
Given $1<p<\infty$, we say that $\Phi$ satisfies the $B_p$ condition,
denoted by $\Phi\in B_p$, if
\[ \int_1^\infty \frac{\Phi(t)}{t^p}\,\frac{dt}{t} <  \infty.  \]

Given a cube $Q$ we define
the localized Orlicz norm $\|f\|_{\Phi,Q}$  by
$$\|f\|_{\Phi,Q}=\inf\bigg\{ \lambda>0: \dashint_\Omega\Phi\bigg(\frac{|f(x)|}{\lambda}\bigg)\,dx\leq 1\bigg\}<\infty.$$
The pair $\Phi,\,\bar{\Phi}$ satisfy the generalized H\"older
inequality in the scale of Orlicz spaces:
\begin{equation}  \label{eqn:gen-holder}
\avgint_Q |f(x)g(x)|\,dx \leq 2\|f\|_{\Phi,Q}\|g\|_{\Psi,Q}.
\end{equation}

P\'erez proved that if the term on the right in the two weight $A_p$
condition is ``bumped'' in the scale of Orlicz spaces, then the
maximal operator satisfies a two weight inequality.  Recall that the
Hardy-Littlewood maximal operator is defined by
\[ Mf(x) = \sup_Q \avgint_Q |f(y)|\,dy \cdot \chi_Q(x). \]

\begin{theorem} \label{thm:perez-max}
Given $1<p<\infty$, suppose $\Phi$ is a Young function such that
$\bar{\Phi}\in B_{p}$.  If $(u,v)$ is a pair of weights such that
\[ \sup_Q \|u^{\frac{1}{p}}\|_{p,Q}\|v^{-\frac{1}{p}}\|_{\Phi,Q} <
  \infty, \]
then $M :
L^p(v)\rightarrow L^p(u)$.
\end{theorem}

\begin{remark}
For instance, if we take $\Phi(t)=t^{p'}\log(e+t)^{p'-1+\delta}$,
$\delta>0$, then $\bar{\Phi}(t)\approx t^p\log(e+t)^{-1-\epsilon}$,
$\epsilon>0$, and $\bar{\Phi}\in B_p$.   Orlicz functions of this kind are referred to as ``log
bumps.''
\end{remark}

\medskip

It was conjectured (see~\cite{cruz-uribe-perez02}) that a comparable
result held for Calder\'on-Zygmund singular integral operators if both
terms in the two weight $A_p$ condition were bumped.  After a number
of partial results, this was proved by Lerner~\cite{Lern2012}.  Recall
that a Calder\'on-Zygmund
singular integral is an operator $T : L^2 \rightarrow L^2$ such
that if $f\in C_c^\infty(\Rd)$,  then for $x\not\in \supp(f)$,
\[ Tf(x) = \int_{\R^d} K(x,y)f(y)\,dy, \]
where the kernel $K : \R^{d}\times \R^d \setminus \Delta \rightarrow
\C$ ($\Delta = \{ (x,x) : x \in \R^d \}$), satisfies
\[ |K(x,y)| \leq \frac{C}{|x-y|^d}, \]
and
\[ |K(x,y)-K(x,y+h)| + |K(x,y)-K(x+h,y)| \leq C
  \frac{|h|^\delta}{|x-y|^{d+\delta}}, \]
for some $\delta>0$ and $|x-y|>2|h|$.

\begin{theorem} \label{thm:lerner-sio}
Given $1<p<\infty$, suppose $\Phi$ and $\Psi$ are Young functions such
that $\bar{\Phi}\in B_{p}$ and $\bar{\Psi}\in B_{p'}$.  If $(u,v)$ is
a pair of weights such that
\[ \sup_Q \|u^{\frac{1}{p}}\|_{\Psi,Q}\|v^{-\frac{1}{p}}\|_{\Phi,Q} <
  \infty, \]
and if $T$ is a Calder\'on-Zygmund singular integral, then $T:
L^p(v)\rightarrow L^p(u)$.
\end{theorem}

Analogous results hold for the fractional maximal operator $M_\alpha$,
and the fractional integral operator $I_\alpha$, $0<\alpha<d$, defined
by
\[ M_\alpha f(x) = \sup_Q |Q|^{\frac{\alpha}{d}}\avgint_Q |f(y)|\,dy
  \cdot \chi_Q(x), \]
and
\[ I_\alpha f(x) = \int_{\R^d} \frac{f(y)}{|x-y|^{d-\alpha}}\,dy. \]
For these operators we are interested in off-diagonal inequalities,
when $1<p\leq q < \infty$.  The corresponding two weight condition is
\[ [u,v]_{A_{p,q}^\alpha} =
\sup_Q |Q|^{\frac{\alpha}{d}-\frac{1}{q}+\frac{1}{p}}
\|u^{\frac{1}{q}}\|_{q,Q}\|v^{-\frac{1}{p}}\|_{p',Q} < \infty. \]
(In the one weight case, which requires
$\frac{1}{p}-\frac{1}{q}=\frac{\alpha}{d}$, the weight $w$ satisfies
$u=w^q$, $v=w^p$.  See~\cite{MR2797562} for details.)  Again, this
condition is itself not sufficient, but if the norms are bumped a
sufficient condition is gotten.  For the off-diagonal inequalities
(i.e., when $p<q$) we replace the $B_p$ condition by the weaker
$B_{p,q}$ condition:  we say a Young function $\Phi\in B_{p,q}$ if
\[ \int_1^\infty \frac{\Phi(t)^{\frac{q}{p}}}{t^q}\frac{dt}{t} <
    \infty. \]
It was shown in \cite{CM} that $B_p\subsetneq B_{p,q}$ when $p<q$.  The following two results were first proved by P\'erez~\cite{P} with
the stronger $B_p$ condition; they were improved to use the $B_{p,q}$
condition in~\cite{CM}.

\begin{theorem} \label{thm:frac-max}
Given $1< p \leq q<\infty$ and $0<\alpha<d$, suppose $\Phi$ is a Young
function such that $\bar{\Phi}\in B_{p,q}$.  If $(u,v)$ is a pair of
weights such that
\[ \sup_Q |Q|^{\frac{\alpha}{d}-\frac{1}{q}+\frac{1}{p}}
\|u^{\frac{1}{q}}\|_{q,Q}\|v^{-\frac{1}{p}}\|_{\Phi,Q} < \infty, \]
then $M_\alpha : L^p(v) \rightarrow L^q(u)$.
\end{theorem}

\begin{theorem} \label{thm:fract-int}
Given $1< p \leq q<\infty$ and $0<\alpha<d$, suppose $\Phi$ and $\Psi$
are Young
functions such that $\bar{\Phi}\in B_{p,q}$ and $\bar{\Psi}\in B_{q',p'}$.  If $(u,v)$ is a pair of
weights such that
\[ \sup_Q |Q|^{\frac{\alpha}{d}-\frac{1}{q}+\frac{1}{p}}
\|u^{\frac{1}{q}}\|_{\Psi,Q}\|v^{-\frac{1}{p}}\|_{\Phi,Q} < \infty, \]
then $I_\alpha : L^p(v) \rightarrow L^q(u)$.
\end{theorem}

\bigskip

The primary goal of this paper is to generalize
Theorems~\ref{thm:perez-max} through~\ref{thm:fract-int} to the
setting of matrix weights.  To state our results we first give some
basic information on matrix weights.  For more details,
see~\cite{CRM,G,MR1928089}.  A matrix weight $U$ is an $n\times n$
self-adjoint matrix function with locally integrable entries such that
$U(x)$ is positive definite for a.e. $x\in \R^d$.  For a matrix weight
we can define $U^r$ for any $r\in \R$, via diagonalization.
Given an exponent $1 \leq p< \infty$ and an $n\times n$ matrix weight
$U$ on $\R^d$ we define the matrix weighted space $L^p(U)$ to be the
set of measurable, vector-valued functions $f:\R^d \rightarrow \Cn$
such that
$$\|{f}\|_{L^p(U)}=\left(\int_{\R^d} |U(x)^{\frac1p} f(x)|^p\,dx\right)^{\frac{1}{p}}<\infty.$$

Given a matrix weight $U$ and $x\in \R^d$,  define the operator norm of $U(x)$ by
\[ |U(x)|_\op = \sup_{\substack{{e}\in \Cn\\|{e}|=1}}|U(x){e}|.  \]
For brevity, given a norm $\|\cdot\|$ on a some scalar valued Banach
function space (e.g., $L^p$), we will write $\|U\|$ for
$\||U|_\op\|$ and $\|Ue\|$ for $\| |Ue| \|$.

Given two matrix weights $U$ and $V$, a linear operator $T$  satisfies
$$T:L^p(V)\rightarrow L^q(U)$$
if and only if
$$U^{\frac1q}TV^{-\frac1p}:L^p(\R^d,\Cn)\ra L^q(\R^d,\Cn),$$
and it is in this form that we will prove matrix weighted norm
inequalities.  However, this approach no longer works for sublinear
operators such as maximal operators.  Following the approach
introduced in~\cite{CG,G} we define a matrix weighted fractional maximal
operator.  Given matrix weights $U$ and $V$ and $0<\alpha<d$, we
define
\begin{equation} \label{eqn:matrix-max}
M_{\al,U,V} f(x)=\sup_{Q\ni x}\frac{1}{|Q|^{1-\frac{\al}{d}}}\int_Q
|U(x)^{\frac1q}V(y)^{-\frac1p}f(y)|\,dy.
\end{equation}
When $U=V$, this operator was first considered
in~\cite{Isralowitz:2016we}.

Our first result give sufficient conditions on the matrices $U$ and
$V$ for $M_{\alpha,U,V}$ to be bounded from $L^p(\R^d,\Cn)$ to
$L^q(\R^d,\Cn)$.

\begin{theorem} \label{thm:mainmax}
  Given $0 \leq \alpha<d$ and $1<p\leq q<\infty$ such that
  $\frac{1}{p}-\frac{1}{q} \leq \frac{\alpha}{d}$, suppose $\Phi$ is a
  Young function with $\bar{\Phi}\in B_{p,q}$.  If $(U,V)$ is a pair
  of matrix weights such that
\begin{equation}\label{matrixbump1}
[U,V]_{p,q,\Phi} = \sup_Q  |Q|^{\frac{\al}{d}+\frac1q-\frac1p} \left(\,\dashint_Q
  \|U(x)^{\frac1q}V^{-\frac1p}\|_{\Phi,
    Q}\,dx\right)^{\frac{1}{q}}<\infty,\end{equation}
then $M_{\al,U,V}:L^p(\R^d,\Cn)\ra L^q(\R^d,\Cn)$.
\end{theorem}

\begin{remark}
In the scalar case (i.e., when $n=1$) Theorem~\ref{thm:mainmax}
immediately reduces to Theorem~\ref{thm:perez-max} when $\alpha=0$ and
Theorem~\ref{thm:frac-max} when $\alpha>0$.
\end{remark}

\begin{remark} \label{remark:one-weight}
Theorem~\ref{thm:mainmax} generalizes two results known in the one
weight case (i.e., when $U=V$).  When $p=q$ and $\alpha=0$, if we take
$\Phi(t)=t^{p'}$, then  the condition~\eqref{matrixbump1} reduces to the matrix
$A_p$ condition,
\begin{equation} \label{eqn:one-wt-Ap}
[U]_{A_p} = \sup_Q \avgint_Q \bigg( \avgint_Q
|U(x)^{\frac{1}{p}}U(y)^{-\frac{1}{p}}|_\op^{p'}\,dy\bigg)^{\frac{p}{p'}}
\,dx<\infty,
\end{equation}
 which is sufficient for $M_U=M_{0,U,U}$ to be bounded
on $L^p(\R^d,\Cn)$:  see~\cite{CG,G}.

Similarly, when $\alpha>0$ and $\frac{1}{p}-\frac{1}{q}=\frac{\alpha}{d}$, and we
again take $\Phi(t)=t^{p'}$, then ~\eqref{matrixbump1} becomes the matrix
$A_{p,q}$ condition,
\[ [U]_{A_{p,q}} = \sup_Q \avgint_Q \bigg( \avgint_Q
|U(x)^{\frac{1}{q}}U(y)^{-\frac{1}{p}}|_\op^{p'}\,dy\bigg)^{\frac{q}{p'}}
\,dx<\infty, \]
 introduced in~\cite{Isralowitz:2016we}, where they
showed this condition is
sufficient for $M_{\alpha,U}=M_{\alpha,U,U}$ to map $L^p(\R^d,\Cn)$
into $L^q(\R^d,\Cn)$.
\end{remark}

\begin{remark}
In Theorem~\ref{thm:mainmax} the restriction on $p$ and $q$
that $\frac{1}{p}-\frac{1}{q}\leq \frac{\alpha}{d}$ is natural. For  if the opposite
inequality holds, given matrix weights $U$ and $V$ such that
$M_{\al,U,V}:L^p(\R^d,\Cn)\ra L^q(\R^d,\Cn)$, then $U(x)=0$ almost
everywhere.  See Proposition~\ref{prop:q-upper-bound} below.  In the
scalar case, this was first proved by Sawyer~\cite{Sawyer:1982bt}.
\end{remark}

\medskip

Our second result gives sufficient conditions on the matrices $U$ and
$V$ for $I_\alpha$ to map $L^p(V)$ to $L^q(U)$.   Here and in
Theorem~\ref{thm:CZO}, by $\|\cdot\|_{\Phi_y,Q}$ we mean that the Orlicz
norm is taken with respect to the $y$ variable.  We define
$\|\cdot\|_{\Psi_x,Q}$ similarly.

 \begin{theorem} \label{thm:mainint}
   Given $0<\alpha<d$ and $1<p\leq q<\infty$ such that
   $\frac{1}{p}-\frac{1}{q}\leq \frac{\alpha}{d}$, suppose that $\Phi$ and $\Psi$ are Young
   functions with $\bar{\Phi}\in B_{p,q}$ and $\bar{\Psi}\in B_{q'}$.
   If $(U,V)$ is a pair of matrix weights such that
\begin{equation} \label{matrixbump2}
[U,V]_{p,q,\Phi,\Psi}= \sup_Q |Q|^{\frac{\al}{d}+\frac1q-\frac1p}
\Big\|\|U(x)^{\frac1q}V(y)^{-\frac1p}\|_{\Phi_y,Q}\Big\|_{\Psi_x,Q}<\infty,
\end{equation}
then $I_{\al}:L^p(V) \rightarrow L^q(U)$.
\end{theorem}

\begin{remark}
In the scalar case, Theorem~\ref{thm:mainint} reduces to a special case of
Theorem~\ref{thm:fract-int} in that we do not recapture the weaker
hypothesis $\bar{\Psi}\in B_{q',p'}$.  This is a consequence of
our proof; we conjecture that this result remains true with this
weaker hypothesis.
\end{remark}

\begin{remark}
In the one weight case, it was proved in~\cite{Isralowitz:2016we} that
if $\frac{1}{p}-\frac{1}{q}=\frac{\alpha}{d}$ and $U\in A_{p,q}$, then
$I_\alpha : L^p(U) \rightarrow L^q(U)$.
\end{remark}

\begin{remark}
As for the fractional maximal operator, the restriction that
$\frac{1}{p}-\frac{1}{q}\leq \frac{\alpha}{d}$ is natural.  In the
scalar case (i.e., when $n=1$),  if the
opposite inequality holds, then, since $M_\alpha f(x)
\lesssim I_\alpha(|f|)(x)$, we have that the weights are trivial.
See~\cite{Sawyer:1982bt} for details.
\end{remark}

\medskip

Our third result gives sufficient conditions on the matrices $U$ and
$V$ for a Calder\'on-Zygmund operator $T$ to map $L^p(V)$ to
$L^p(U)$.

\begin{theorem} \label{thm:CZO}
Given $1 < p < \infty$,  suppose $\Phi$
  and $\Psi$ are Young functions with $\bar{\Phi}\in B_{p}$ and
  $\bar{\Psi}\in B_{p'}$.  If $(U,V)$ is a pair of matrix weights
such that
\begin{equation} \label{matrixbump3}
[U,V]_{p,\Phi,\Psi} = \sup_Q
\Big\|\|U(x)^{\frac1p}V(y)^{-\frac1p}\|_{\Phi_y,Q}\Big\|_{\Psi_x,Q}<\infty,
\end{equation}
and if $T$ is a   Calder\'on-Zygmund operator, then  $T : L^p(V) \rightarrow
  L^p(U)$.
\end{theorem}

\begin{remark}
Theorem~\ref{thm:CZO} also holds if $T$ is a Haar shift operator or a
paraproduct.  See the discussion in Section~\ref{section:CZO}
below.
\end{remark}

\medskip

As a corollary to Theorem~\ref{thm:CZO} we can prove quantitative
one weight  estimates for Calder\'on-Zygmund operators.   To state our
result, recall that if $W$ is in matrix $A_p$, then for every $e\in
\Cn$, $|W^\frac{1}{p} {e}|^p$ is a scalar $A_p$ weight, and
\[ [|W^\frac{1}{p} {e}|^p ]_{A_p} \leq [W]_{A_p}. \]
Thus, following~\cite{Nazarov:2017ufa}, we can then define the
``scalar $A_\infty$'' constant of $W$ by
\[ [W]_{A_{p,\infty}^{\sca} }
 = \sup_{e \in \Cn} [|W^{\frac{1}{p}}e|^p ]_{A_\infty}. \]
(We will make precise our definition of $A_\infty$ in  Section~\ref{section:CZO}.)

\begin{corollary} \label{SharpSparseCor}
Given $1<p<\infty$, suppose $W$ is a matrix $A_p$ weight.  If $T$ is a
Calder\'on-Zygmund operator, then
\[ \|T\|_{L^p(W) } \lesssim [W]_{A_p}^{\frac{1}{p}} \,
  [W^{-\frac{p'}{p}}]_{A^{\sca}_{p',\infty}}^{\frac{1}{p}}\,
 [W] _{A^{\sca}_{p,\infty}} ^{\frac{1}{p'}} \lesssim [W]_{A_p}^{1+\frac{1}{p-1}-\frac{1}{p}}. \]
 \end{corollary}

\begin{remark}
  Corollary~\ref{SharpSparseCor} appears to be the first quantitative
  estimate for matrix weighted inequalities for singular integrals
  for all $p$, $1<p<\infty$.
  Qualitative one weight, matrix $A_p$ estimates for
  Calder\'on-Zygmund operators were first proved in~\cite{CG,G}.
  Bickel, Petermichl and Wick~\cite{MR3452715} proved that for the
  Hilbert transform $H$,
  $\|H\|_{L^2(W)} \lesssim [W]_{A_2}^{\frac{3}{2}}\log([W]_{A_2})$.
  This result was improved by Nazarov, {\em et
    al.}~\cite{Nazarov:2017ufa} and Culiuc, di Plinio and
  Ou~\cite{Culiuc:2016vp} and extended it to all Calder\'on-Zygmund
  operators $T$, getting
  $\|T\|_{L^2(W)} \lesssim [W]_{A_2}^{\frac{3}{2}}$.  (In fact,
  in~\cite{Nazarov:2017ufa} they prove a stronger result which we will
  discuss below.)  Corollary~\ref{SharpSparseCor} reduces to this
  estimate when $p=2$.

We doubt that our estimate is sharp:  it is reasonable to conjecture
that the sharp exponent for matrix weights is the same as in the
scalar case:  $\max\{1,\frac{1}{p-1}\}$.  We do note that in the
scalar case, our exponent is sharper than what would be gotten from
Rubio de Francia extrapolation, which starting from the exponent
$\frac{3}{2}$ when $p=2$ is  $\frac{3}{2}\max\{1,\frac{1}{p-1}\}$.  In
particular, it is asymptotically sharp as $p\rightarrow \infty$.
\end{remark}

\medskip

We now consider the two weight
matrix $A_{p,q}$ condition,
\begin{equation} \label{eqn:matrix-Apq}
[U,V]_{A_{p,q}^\alpha } = \sup_Q
|Q|^{\frac{\alpha}{d}+\frac{1}{q}-\frac{1}{p}}
\left(\avgint_Q \bigg( \avgint_Q
|U(x)^{\frac{1}{q}}V(y)^{-\frac{1}{p}}|_\op^{p'}\,dy\bigg)^{\frac{q}{p'}}
\,dx\right)^{\frac{1}{q}}<\infty.
\end{equation}
By the properties of Orlicz norms, we have that
$[U,V]_{A_{p,q}^\alpha }$ is dominated by $[U,V]_{p,q,\Phi}$ and
$[U,V]_{p,q,\Phi,\Psi}$.  As we noted in
Remark~\ref{remark:one-weight} above, this condition is sufficient in
the one weight case for the strong type, two weight norm inequalities
for maximal and fractional integrals.  However, even in the scalar
case this condition is not sufficient for two weight norm inequalities
for fractional maximal or integral operators~\cite{CruzUribe:2016ji}.
It is known to be necessary and sufficient for averaging operators to
map $L^p(v)$ into $L^p(u)$~\cite{MR1622773} and for the fractional
maximal operator to map $L^p(v)$ into
$L^{q,\infty}(u)$~\cite{MR2797562}.  We give two generalizations of
these results to the matrix setting.   Since these results include
endpoint estimates, we extend the definition of $A_{p,q}$ to the case
$p=1$:  given matrix weights $U$ and $V$, define
\begin{equation} \label{eqn:matrixA1q}
 [U,V]_{1,q}^\alpha
= \sup_Q |Q|^{\frac{\alpha}{d}+\frac{1}{q}-1}
\esssup_{y\in Q} \left(\avgint_Q
  \OP{U^{\frac{1}{q}}(x)V^{-1}(y)}\,dx\right)^{\frac{1}{q}}
 < \infty.
\end{equation}

Our first result concerns averaging operators.  For $0\leq \alpha <d$,
given a cube $Q$, define
\[ A_Q^\alpha f(x) = |Q|^{\frac{\alpha}{d}} \avgint_Q f(y)\,dy \cdot \chi_Q(x). \]
More generally, given a family $\Q$ of disjoint cubes, define
\[ A_\Q^\alpha f(x) = \sum_{Q\in \Q} A_Q^\alpha f(x). \]

\begin{theorem} \label{theorem:avg-op}
Given $0\leq \alpha<d$,  $1\leq p\leq q<\infty$ such that
$\frac{1}{p}-\frac{1}{q}\leq \frac{\alpha}{d}$, and a pair of matrix weights $(U,V)$,  the following are equivalent:
\begin{enumerate}

\item $(U,V) \in A_{p,q}^\alpha$;

\item Given any set $\Q$ of pairwise disjoint cubes in $\Rd$,
\[ \|A_\Q^\alpha f\|_{L^q(U)} \lesssim
  [U,V]_{A_{p,q}^\alpha}\|f\|_{L^p(V)}, \]
where the constant is independent of $\Q$.
\end{enumerate}
\end{theorem}

\begin{remark}
  In the one weight, scalar case when $p=q$ Theorem~\ref{theorem:avg-op} was
  implicit in Jawerth~\cite{jawerth86}; for the general result in the
  scalar case, see Berezhno{\u\i}~\cite{MR1622773}.  In the one weight
  matrix case, again when $p=q$, Theorem~\ref{theorem:avg-op} was
  proved in~\cite{CRM}.
\end{remark}

\begin{remark}
As a corollary to Theorem~\ref{theorem:avg-op} we prove two weight
estimates for convolution operators and approximations of the
identity, generalizing one weight results from~\cite{CRM}.  See
Corollary~\ref{cor:convolution-op} below.
\end{remark}
\medskip

Our second result is a weak type inequality for a two weight variant
of the so-called auxiliary maximal operator introduced
in~\cite{CG,G}.  Given $0\leq \alpha <d$ and matrix weights $U$ and
$V$, define
\begin{equation}  \label{eqn:two-wt-auxiliary}
M_{\alpha, U, V}' f(x) = \sup_{Q} {|Q| ^{ \frac{\alpha}{d}}}
\avgint_Q |\MC{U}_Q ^q V^{-\frac{1}{p}} (y) {f}(y)| \, dy \cdot \chi_Q(x),
\end{equation}
where $\Uu_Q^q$ is the reducing operator associated with the matrix
$U$. (For a precise definition, see Section~\ref{section:prelim}
below.)  Given any cube $Q$, the associated averaging operator is
\[ B_Q^\alpha f(x) =
{|Q| ^{ \frac{\alpha}{d}}}
\avgint_Q |\MC{U}_Q ^q V^{-\frac{1}{p}} (y) {f}(y)| \, dy \cdot
\chi_Q(x). \]

\begin{theorem} \label{thm:mainweak}
  Given $0\leq \alpha<d$, $1\leq p\leq q<\infty$ such that
  $\frac{1}{p}-\frac{1}{q}\leq \frac{\alpha}{d}$, and a pair of matrix
  weights $(U,V)$, the following are equivalent:
\begin{enumerate}
\item $ \ (U, V) \in A_{p, q}^\alpha$;
\item $\ M_{\alpha, U, V} ' : L^p \rightarrow L^{q, \infty}$;
\item For every cube $Q$, $B_Q^\alpha : L^p \rightarrow L^{q,\infty}$ with norm
  independent of $Q$.
\end{enumerate}
\end{theorem}

\begin{remark}
It is very tempting to conjecture that Theorem~\ref{thm:mainweak}
remains true with the auxiliary maximal operator replaced by
$M_{\alpha, U,V}$, but we have been unable to prove this.  We can
prove~\ref{thm:mainweak} because the
auxiliary maximal operator is much easier to work with when
considering weak type inequalities.
\end{remark}

\medskip

Finally, as a corollary to Theorem~\ref{thm:mainint} we prove a
``mixed'' Poincar\'e inequality involving both scalar and matrix
weights.

\begin{theorem} \label{thm:poincare}
Given $1<p \leq q< \infty$ such that
$\frac{1}{p}-\frac{1}{q}\leq \frac{1}{d}$, suppose that $\Phi$ and
$\Psi$ are Young functions with $\bar{\Phi}\in B_{p,q}$ and
$\bar{\Psi} \in B_{q'}$.  If $u$ is a scalar weight and $V$ is a
matrix weight such that
\begin{equation} \label{eqn:poincare1}
 \sup_Q |Q|^{\frac{1}{d}+\frac{1}{q}-\frac{1}{p}}
\|u^{\frac{1}{q}}\|_{\Psi,Q} \|V^{-\frac{1}{p}}\|_{\Phi,Q} < \infty,
\end{equation}
then given any open convex set $E\subset \R^d$ with $u(E)<\infty$, and
any scalar function $f\in C^1(E)$,
\begin{equation} \label{eqn:poincare2}
 \left(\int_E |f(x)-f_{E,u}|^q u(x)\,dx\right)^{\frac{1}{q}}
\lesssim \left(\int_E |V^{\frac{1}{p}}(x) \nabla
    f(x)|^p\,dx\right)^{\frac{1}{p}},
\end{equation}
where $f_{E,u} = u(E)^{-1}\int_E f(x) u(x)\,dx$.  The implicit
constant is independent of $E$.
\end{theorem}

\begin{remark}
Poincar\'e inequalities of this kind play a role in the study of
degenerate elliptic equations.  See, for
instance,~\cite{MR2906551,MR3369270,MR3388872,MR2204824,MR2574880}.
As an immediate consequence of Theorem~\ref{thm:poincare} we can use
the main result in~\cite{CruzUribe:2017wz} to prove the existence of
weak solutions to a Neumann boundary value problem for a degenerate
$p$-Laplacian.  See Corollary~\ref{cor:p-laplacian} below.
\end{remark}

\medskip

The remainder of this paper is organized as follows.  In
Section~\ref{section:prelim} we gather together some preliminary
results about the so called
reducing operators associated with matrix weights.  Reducing operators
play a major role in all of our proofs.

In Section~\ref{section:proof-mainmax} we prove
Theorem~\ref{thm:mainmax} and Proposition~\ref{prop:q-upper-bound}
In Section~\ref{section:proof-mainint} we prove
Theorem~\ref{thm:mainint}.  In our proofs of these two theorems we make extensive use of the
theory of dyadic approximations for fractional maximal and integral
operators; for the scalar theory, see~\cite{CruzUribe:2016ji}.

In Section~\ref{section:CZO} we prove Theorem~\ref{thm:CZO} and
Corollary~\ref{SharpSparseCor}.  In our proof we use the recent result
of Nazarov, {\em et al.}~\cite{Nazarov:2017ufa}, who extended dyadic approximation
theory for singular integrals to the matrix setting, and showed that
to prove matrix weighted estimates for Calder\'on-Zygmund
  operators it is enough to prove them for sparse operators.

In Section~\ref{section:characterization}  we prove
Theorems~\ref{theorem:avg-op} and~\ref{thm:mainweak}, and prove
Corollary~\ref{cor:convolution-op} about convolution operators.
Finally, in Section~\ref{section:poincare} we prove
Theorem~\ref{thm:poincare}, and prove Corollary~\ref{cor:p-laplacian} giving
weak solutions to a degenerate $p$-Laplacian.

Throughout this paper notation is standard or will be defined as
needed.  If we write $X\lesssim Y$, we mean that $X\leq cY$, where the
constant $c$ can depend on the dimension $d$ of the underlying space
$\R^d$,  the dimension $n$ of our vector functions, the exponents $p$
and $q$
in the weighted Lebesgue spaces, and the
underlying fractional maximal or integral operators (i.e., on
$\alpha$) or on the underlying Calder\'on-Zygmund
  operator.  The dependence on the matrix weights will always be made
  explicit.  If we write $X\approx Y$, then $X\lesssim Y$ and
  $Y\lesssim X$.

\section{Reducing operators}
\label{section:prelim}

Given a matrix weight $A$, a Young function $\Psi$, and a cube
$Q$, we can define a norm on $\C^n$ by
$\|Ae\|_{\Psi,Q}$, $e \in \C^n$.
The following lemma yields a very important tool in the study of
matrix weights, the so-called reducing operator, which lets us replace
this norm by a norm induced by a constant positive matrix.    The following result
was proved by Goldberg~\cite[Proposition~1.2]{G}.

\begin{lemma} \label{lemma:reducing-vec}
Given a matrix weight $A$, a Young function $\Psi$, and a cube
$Q$, there exists a (constant positive) matrix $\A_Q^\Psi$, called a
reducing operator of $A$, such that for
all $e\in \C^n$,
\[ |\A_Q^\Psi e| \approx \|Ae\|_{\Psi,Q}, \]
where the implicit constants depend only on $d$.
\end{lemma}

As a consequence of Lemma~\ref{lemma:reducing-vec}, we get the
following result for the norms of reducing operators.  These estimates
are implicit in the literature, at least for $L^p$ norms; we prove them  for the
convenience of the reader.

\begin{proposition} \label{prop:reducing-norm}
Given matrix weights $A$ and $B$,  Young functions $\Phi$ and $\Psi$, a cube
$Q$, and reducing operators $\A_Q^\Psi$ and $\B_Q^\Phi$, then for all $e\in
\C^n$,
\begin{equation} \label{eqn:reducing-norm1}
|\A_Q^\Psi|_\op \approx \|A\|_{\Psi,Q},
\end{equation}
\begin{equation} \label{eqn:reducing-norm2}
|\A_Q^\Psi\B_Q^\Phi|_\op \approx \|A(x) \B_Q^\Phi \|_{\Psi_x,Q}
\approx \Big\| \|A(x)B(y)\|_{\Phi_y} \Big\|_{\Psi_x,Q}.
\end{equation}
In both cases the implicit constants depend only on $d$.
\end{proposition}

\begin{remark}
As will be clear from the proof, the first estimate
in~\eqref{eqn:reducing-norm2} is true if $\B_Q^\Phi$ is replaced with
any constant matrix.
\end{remark}

\begin{proof}
To prove \eqref{eqn:reducing-norm1} fix an orthonormal basis
$\{e_j\}_{j=1}^n$ of $\C^n$.  Then by the definition of the operator norm and of
reducing operators,
\[ |\A_Q^\Psi|_\op \approx \sum_{j=1}^n|\A_Q^\Phi e_j|
\approx \sum_{j=1}^n\|Ae_j\|_{\Phi,Q}
\approx \|A\|_{\Phi,Q}.  \]

\medskip

The proof of~\eqref{eqn:reducing-norm2} is similar, but we exploit the
fact that while matrix products of self-adjoint matrices do not
commute, they have the same operator norm:
\begin{align*}
|\A_Q^\Psi \B_Q^\Phi|_\op
& \approx \sum_{j=1}^n |\A_Q^\Psi \B_Q^\Phi e_j| \\
& \approx \sum_{j=1}^n \|A(x) \B_Q^\Phi e_j\|_{\Psi_x,Q} \\
& \approx \|A(x) \B_Q^\Phi \|_{\Psi_x,Q} \\
& = \| \B_Q^\Phi A(x) \|_{\Psi_x,Q} \\
& \approx \sum_{j=1}^n \| \B_Q^\Phi A(x) e_j\|_{\Psi_x,Q} \\
& \approx \sum_{j=1}^n \Big\| \|B(y)A(x) e_j\|_{\Phi_y,Q}
\Big\|_{\Psi_x,Q} \\
& \approx \Big\| \|B(y)A(x) \|_{\Phi_y,Q} \Big\|_{\Psi_x,Q} \\
& =  \Big\| \|A(x)B(y) \|_{\Phi_y,Q} \Big\|_{\Psi_x,Q}.
\end{align*}
\end{proof}

As a consequence of Proposition~\ref{prop:reducing-norm} we can
restate all of the weight conditions in our theorems in terms of
reducing operators.  Given matrix weights $U$ and $V$,  Young
functions  $\Psi$ and $\Phi$, and
$1 \leq p\leq q<\infty$, let $\Uu_Q^{q,\Psi}$ and $\Vv_Q^{p,\Phi}$ be
the reducing operators

\[ |\Uu_Q^{q,\Psi}e| \approx \|U^{\frac{1}{q}} e\|_{\Psi,Q},
| \Vv_Q^{p,\Phi}e| \approx \|V^{-\frac{1}{p}} e \|_{\Phi,Q}.
 \]
If $\Psi(t)=t^q$ or $\Phi(t)=t^{p'}$ then we will write $\Uu_Q^{q}$,
$\Vv_Q^{p,p'}$ (or more simply, $\Uu_Q^q$, $\Vv_Q^{p'}$).

With this definition, we have the following equivalences: in
Theorem~\ref{thm:mainmax},
\begin{equation} \label{eqn:mainmax-alt}
[U,V]_{p,q,\Phi} \approx
\sup_Q|Q|^{\frac{\alpha}{d}+\frac{1}{q}-\frac{1}{p}}
|\Uu_Q^{q,q}\Vv_Q^{p,\Phi}|_\op;
\end{equation}
in
Theorem~\ref{thm:mainint},
\begin{equation} \label{eqn:mainint-alt}
[U,V]_{p,q,\Phi,\Psi} \approx
\sup_Q|Q|^{\frac{\alpha}{d}+\frac{1}{q}-\frac{1}{p}}
|\Uu_Q^{q,\Psi}\Vv_Q^{p,\Phi}|_\op;
\end{equation}
in
Theorem~\ref{thm:CZO},
\begin{equation} \label{eqn:CZO-alt}
[U,V]_{p,\Phi,\Psi} \approx \sup_Q |\Uu_Q^{p,\Psi}\Vv_Q^{p,\Phi}|_\op.
\end{equation}
When $p>1$  we can restate the two weight $A_{p,q}$
condition~\eqref{eqn:matrix-Apq} as
\begin{equation} \label{eqn:Apq-alt}
[U,V]_{A_{p,q}^\alpha} \approx
\sup_Q|Q|^{\frac{\alpha}{d}+\frac{1}{q}-\frac{1}{p}}
|\Uu_Q^{q}\Vv_Q^{p'}|_\op,
\end{equation}
and when $p=1$ by
\begin{equation} \label{eqn:A1q-alt}
 [U,V]_{A_{1,q}^\alpha} \approx
 \sup_Q|Q|^{\frac{\alpha}{d}+\frac{1}{q}-1}
\esssup_{y\in Q}|\Uu_Q^{q}V^{-1}(y)|_\op.
\end{equation}
Finally, we will need the following lemma in the proof of
Corollary~\ref{SharpSparseCor}.  It is a quantitative version of a
result proved in Roudenko~\cite[Corollary~3.3]{MR1928089}.  It follows
at once if we use \eqref{eqn:Apq-alt} to restate the definitions of
one weight matrix $A_p$ and $A_{p'}$ from~\eqref{eqn:one-wt-Ap}.

\begin{lemma} \label{lemma:matrix-dual}
Given $1<p<\infty$ and a matrix weight $W$, if $W\in A_p$, then
$W^{-\frac{p'}{p}} \in A_{p'}$ and
\[ [W]_{A_p}^{\frac{1}{p}} \approx
[ W^{-\frac{p'}{p}}]_{A_{p'}}^{\frac{1}{p'}}. \]
\end{lemma}

\section{Proof of Theorem~\ref{thm:mainmax}}
\label{section:proof-mainmax}

We first prove that in Theorem \ref{thm:mainmax} we may assume without
loss of generality that $\frac{1}{p}-\frac{1}{q} \leq
\frac{\alpha}{d}$.

\begin{proposition} \label{prop:q-upper-bound}
Given $0<\alpha<d$, matrix weights $U$, $V$ and $1<p<q<\infty$ such that
  $\frac{1}{p} - \frac{1}{q} > \frac{\alpha}{d}$,
suppose that
  $M_{\alpha,U,V} : L^p\rightarrow L^q$.
If $V^{\frac{1}{p}}$ is
  locally integrable,  then $U(x)=0$ for almost every
  $x\in \Rd$.
\end{proposition}

\begin{proof}
  Fix $Q$ and a vector $e$, and let
  $f(y)=V^{\frac{1}{p}} (y) e\chi_Q(y)$.  Then for $x\in Q$,
\[ M_{\alpha,U,V}f(x)
\geq |Q|^{\frac{\alpha}{d}}\avgint_Q |U^{\frac{1}{q}}(x)e|\,dy
= |Q|^{\frac{\alpha}{d}}|U^{\frac{1}{q}}(x)e|. \]
Therefore,
\[ |Q|^{\frac{\alpha}{d}} \left(\int_Q|U^{\frac{1}{q}}(x)e|^q\,dx\right)^{\frac{1}{q}}
\leq \|M_{\alpha,U,V} f\|_{L^q}
\lesssim \|f\|_{L^p}
= \left(\int_Q |V^{\frac{1}{p}}(x)e|^p\,dx\right)^{\frac{1}{p}}, \]
which in turn implies that
\[ \left(\avgint_Q|U^{\frac{1}{q}}(x)e|^q\,dx\right)^{\frac{1}{q}}
\lesssim |Q|^{\frac{1}{p}-\frac{1}{q}-\frac{\alpha}{d}}
\left(\avgint_Q |V^{\frac{1}{p}}(x)e|^p\,dx\right)^{\frac{1}{p}}.  \]

Let $x_0$ be any Lebesgue point of the functions
$|U^{\frac{1}{q}}(x)e|^q$ and $|V^{\frac{1}{p}}(x)e|^p$ and let $Q_k$
be an sequence of cubes centered at $x_0$ that shrink to this point.
By the Lebesgue differentiation theorem, since
$\frac{1}{p}-\frac{1}{q}-\frac{\alpha}{d}>0$, the righthand side of
the above inequality tends to 0.  Therefore,
$|U^{\frac{1}{q}}(x_0)e|^q=0$.  Since this is true for every vector
$e$, we have that $|U(x_0)|_\op = |U^{\frac{1}{q}}(x_0)|_\op^q=0$.
Hence, $U(x_0)=0$.
\end{proof}

\medskip

To prove Theorem \ref{thm:mainmax} we will first reduce the problem to
the corresponding dyadic maximal operator.  We recall some facts from
the theory of dyadic operators.  We say that a collection
of cubes $\D$ in $\R^d$ is a dyadic grid if
\begin{enumerate}
\item if $Q\in \D$, then $\ell(Q)=2^k$ for some $k\in \Z$.

\item If $P,\,Q\in \D$, then $P\cap Q \in \{ P,Q,\emptyset\}$.

\item For every $k\in \Z$, the cubes $\D_k=\{ Q\in \D : \ell(Q)=2^k\}$
  form a partition of $\R^d$.
\end{enumerate}
We can approximate arbitrary cubes in $\R^d$ by cubes from a
finite collection of dyadic grids.  (For a proof,
see~\cite[Theorem~3.1]{CruzUribe:2016ji}.)

\begin{proposition}\label{dyadic}
For $t\in \{0,\pm\frac{1}{3}\}^d$ define the sets
 \[ \D^t = \{2^{-k} ([0, 1) ^d + m + t) : k \in \Z, m \in
  \Z^d\}.  \]
Then each $\D^t$ is a dyadic grid, and  given any cube $Q\subset
\R^d$,  there exists $t$ and $Q_t\in \D^t$ such that $Q\subset Q_t$ and
  $\ell(Q_t)\leq 3\ell(Q)$.
\end{proposition}

Given $0\leq \alpha<d$, matrix weights $U$ and $V$ and a dyadic grid
$\D$, define the dyadic maximal operator $M_{\alpha,U,V}^\D$ as
in~\eqref{eqn:matrix-max} but with the supremum taken over all cubes
$Q\in \D$ containing $x$.   Then the following result follows at once
from Proposition~\ref{dyadic}
(cf.~\cite[Proposition~3.2]{CruzUribe:2016ji}).

\begin{proposition} \label{prop:dyadic-max-approx}
Given $0\leq \alpha<d$, matrix weights $U$ and $V$, let $\D^t$ be the
dyadic grids from Proposition~\ref{dyadic}.  Then for all $x\in \R^d$,
\[ M_{\alpha,U,V}f(x) \lesssim \sum_{t\in \{0,\pm\frac{1}{3}\}^d}
M_{\alpha,U,V}^{\D^t}f(x). \]
\end{proposition}

As a consequence of Proposition~\ref{prop:dyadic-max-approx}, to prove
Theorem~\ref{thm:mainmax} it will suffice to prove it for
$M_{\al,U,V}^\D$, where $\D$ is any dyadic grid.  For the remainder of
this section, fix a dyadic grid $\D$.

Our proof is adapted from the proof of the boundedness of the one
weight maximal operator in~\cite{G}.  We begin with two lemmas.   For brevity, we will
write $\Vv_Q^\Phi$ for the reducing operator $\Vv_Q^{p,\Phi}$.  The
first gives a norm inequality for an auxiliary
 fractional maximal operator,
analogous to the operator $M_{W}'$ introduced in~\cite{CG,G}.

\begin{lemma} \label{lemma:auxiliary-max}
Given $0\leq \beta<d$, let $1<p\leq q<\infty$ be such that
$\frac{\beta}{d}=\frac{1}{p}-\frac{1}{q}$.  Let $\Phi$ be a Young
function such that $\bar{\Phi}\in B_{p,q}$.  Given  a matrix weight $V$, define the auxiliary
maximal operator
$$ M^\D_{\beta,V}f(x)
=\sup_{Q\in \D }|Q|^{ \frac{\beta}{d}}
\avgint_Q |(\mathcal{V}_Q^\Phi)^{-1}V(y)^{-\frac1p}f(y)|\,dy\cdot \chi_Q(x).$$
Then $M^\D_{\beta,V}:L^p(\R^d,\Cn)\ra L^q(\R^d,\Cn)$.
\end{lemma}

\begin{proof}
Define the Orlicz fractional maximal operator
\[ M_{\beta,\bar{\Phi}}f(x) = \sup_Q
  |Q|^{\frac{\beta}{d}}\|f\|_{\bar{\Phi},Q}\cdot \chi_Q(x); \]
if $\beta=0$, we write $M_{\bar{\Phi}}= M_{0,\bar{\Phi}}$.  It was shown in~\cite{CM} that
\begin{equation} \label{eqn:frac-orliz-max}
M_{\beta,\bar{\Phi}}:L^p(\R^d)\ra L^q(\R^d).
\end{equation}

Now fix $x\in \R^d$ and $Q\in \D$ containing $x$.  Then by the
generalized H\"older inequality \eqref{eqn:gen-holder} we
have that
\[ |Q|^{\frac{\beta}{d}}\avgint_Q
  |(\mathcal{V}_Q^\Phi)^{-1}V(y)^{-\frac1p}f(y)|\,dy
\lesssim \|(\mathcal{V}_Q^\Phi)^{-1}V^{-\frac1p}\|_{\Phi,Q}
|Q|^{\frac{\beta}{d}}\|f\|_{\bar{\Phi},Q}. \]
By the first inequality in~\eqref{eqn:reducing-norm2} (which holds if
we replace the reducing operator $\B_Q^\Phi$ by any matrix), we have
that for all cubes $Q$,
\begin{equation*}
 \|(\mathcal{V}_Q^\Phi)^{-1}V^{-\frac1p}\|_{\Phi,Q}
=\|V^{-\frac1p} (\mathcal{V}_Q^\Phi)^{-1}\|_{\Phi,Q}
\lesssim |\Vv_Q^\Phi (\mathcal{V}_Q^\Phi)^{-1} |_\op = 1.
\end{equation*}
Therefore, if we combine these two inequalities and take the supremum
over all cubes $Q$ containing $x$, we get that
$M^\D_{\beta,V}f(x)\lesssim M_{\beta,\bar{\Phi}}(|f|)(x)$. The desired
norm inequality follows at once.
\end{proof}

For the second lemma, given a cube $Q\in \D$, let $\D(Q) = \{ P \in \D
: P\subset Q\}$ and define the maximal type operator
\begin{equation} \label{eqn:NQ-defn}
 N_Q(x) = \sup_{R\in \D(Q)}
  |Q|^{\frac{\al}{d}+\frac1q-\frac1p}|U(x)^{\frac1q}\mathcal{V}_R^\Phi|_\op
  \cdot \chi_R(x).
\end{equation}

\begin{lemma} \label{lemma:NQq}
Given a pair of matrix weights  $U$, $V$ that satisfy
\eqref{matrixbump1},  then
\begin{equation}\label{NQineq}
\sup_{Q\in \D} \, \dashint_Q N_Q(x)^q\,dx< \infty.
\end{equation}
\end{lemma}

Lemma~\ref{lemma:NQq} is actually an immediate consequence of
Lemma~\ref{NQOrliczLem} which we will need to prove
Theorem~\ref{thm:mainint}, and so its proof is deferred to the next
section: see Remark~\ref{remark:NQq-proof}.

\begin{proof}[Proof of Theorem \ref{thm:mainmax}]
Fix $\beta$ such that $\frac{\beta}{d}=\frac{1}{p}-\frac{1}{q}$.  Note
that by our assumption on $p$ and $q$, $\beta \geq 0$.  Given any cube $Q$,
\begin{align*}
& |Q|^{\frac{\al}{d}}\avgint_Q |U(x)^{\frac1q}V(y)^{-\frac1p}f(y)|\,dy
  \\
 & \qquad \qquad =
|Q|^{\frac{\al}{d}+\frac1q-\frac1p}    |Q|^{\frac1p-\frac1q}
\avgint_Q |U(x)^{\frac1q}\mathcal{V}_Q^\Phi (\mathcal{V}_Q^\Phi)^{-1}V(y)^{-\frac1p}f(y)|\,dy \\
& \qquad \qquad \leq
  |Q|^{\frac{\al}{d}+\frac1q-\frac1p}|U(x)^{\frac1q}\mathcal{V}_Q^\Phi|_{\text{op}}
|Q|^{\frac{\beta}{d}}\avgint_Q |(\mathcal{V}_Q^\Phi)^{-1}V(y)^{-\frac1p}f(y)|\,dy.
\end{align*}

For every $x\in \R^d$ there exists $Q=Q_x\in \D$ such that
$$M_{\al,U,V} ^\D f(x)
\leq
2|Q|^{\frac{\al}{d}+\frac1q-\frac1p}|U(x)^{\frac1q}\mathcal{V}_Q^\Phi|_{\text{op}}
|Q|^{\frac{\beta}{d}}\avgint_Q|(\mathcal{V}_Q^\Phi)^{-1}V(y)^{-\frac1p}f(y)|\,dy.$$
There exists a unique $j=j_x\in \Z$ such that
\begin{equation}\label{2jmax}
2^j<|Q_x|^{\frac{\beta}{d}}\avgint_{Q_x}|(\mathcal{V}_{Q_x}^\Phi)^{-1}V(y)^{-\frac1p}f(y)|\,dy\leq
2^{j+1}.
\end{equation}

Now for each $j \in \Z$, let  $\mathcal S_j$ be the collection of
cubes $Q = Q_x$ that are maximal with respect to \eqref{2jmax}.   Note
that the cubes in $\mathcal S_j$  are disjoint.  Then
for each $x\in \R^d$ there exists $j\in \Z$ and $S\in \mathcal{S}_j$
such that $x\in Q \subset  S$ and
\[
M_{\al,U,V} ^\D f(x) \leq
  2|Q|^{\frac{\al}{d}+\frac1q-\frac1p}|U(x)^{\frac1q}\mathcal{V}_Q^\Phi|_{\text{op}}
|Q|^{\frac{\beta}{d}}\avgint_Q|(\mathcal{V}_Q^\Phi)^{-1}V(y)^{-\frac1p}f(y)|\,dy
\leq 2^{j+2}N_S(x).
\]
Moreover, we have that
\[ \bigcup_{S\in \mathcal S_j} S \subset \{ x\in \R^d : M_{\beta,V}^\D
  f(x)>2^j \}. \]
 Hence, by Lemmas ~\ref{lemma:auxiliary-max} and  ~\ref{lemma:NQq},  
\begin{align*}
\int_{\R^d}|M_{\al,U,V}^\D f(x)|^q\,dx
&\lesssim \sum_{j\in \Z}2^{jq}\sum_{S\in \mathcal{S}_j} \int_S N_S(x)^q\,dx\\
&\leq \sum_{j\in \Z} 2^{jq}\sum_{S\in \mathcal S_j}|S|\\
&= \sum_{j\in \Z} 2^{jq}\Big|\bigcup_{S\in \mathcal S_j}S\Big|\\
&\leq \sum_{j\in \Z}2^{jq}|\{x:M_{\beta,V} ^\D f(x)>2^j\}|\\
&\lesssim \int_{\R^d} M_{\beta,V} ^\D f(x)^q\,dx\\
&\lesssim\left(\,\int_{\R^d} |f(x)|^p\,dx\right)^{\frac{q}{p}}.
\end{align*}
\end{proof}

\section{Proof of Theorem~\ref{thm:mainint}}
\label{section:proof-mainint}

Throughout this section, for brevity we will write
$\Uu_Q^\Psi=\Uu_Q^{q,\Psi}$ and $\Vv_Q^\Phi=\Vv_Q^{p,\Phi}$.  We begin
with a lemma that extends \cite[Lemma~3.3]{G} to the scale of Orlicz spaces.

\begin{lemma} \label{NQOrliczLem}
Given a pair of matrix weights $U$, $V $ that satisfy \eqref{matrixbump2}, then
\begin{equation*}\sup_{Q\in \D}  \|N_Q\|_{\Psi, Q} < \infty \end{equation*}
 \end{lemma}

\begin{remark} \label{remark:NQq-proof}
Since $\bar{\Psi}\in B_{q'}$, we have that $\bar{\Psi}(t)\lesssim t^{q'}$
and so $t^q \lesssim \Psi(t)$.  Therefore, for any cube $Q\in \D$,
$\|N_Q\|_{q,Q} \lesssim \|N_Q\|_{\Psi,Q}$ (cf.~\cite{MR2797562}), and so Lemma~\ref{lemma:NQq}
follows immediately from Lemma~\ref{NQOrliczLem}.
\end{remark}

\begin{proof}[Proof of Lemma \ref{NQOrliczLem}.]
Fix a cube $Q\in \D$.   We first claim that there exists $C>0$ sufficiently large such that
  if  $\{R_j^1\}$ is the collection
  of maximal dyadic subcubes $R$ of $Q$, if any,
  satisfying
\begin{equation*}
|R| ^{\frac{\alpha}{d} + \frac1q -
      \frac1p}\OP{\MC{V}_R ^\Phi \MC{U}_{Q} ^\Psi} > C,
\end{equation*}
 then
\begin{equation} \label{DecayIneq}
\bigg|\bigcup_{j} R_j^1    \bigg| \leq \frac{1}{2}|Q|.
 \end{equation}

 To see this, note that since $\frac{\al}{d}\geq \frac1p-\frac1q $, by
 inequality~\eqref{eqn:reducing-norm2} we have that
\[
C  < |R_j ^1| ^{\frac{\alpha}{d} + \frac1q -
     \frac1p} \OP{\MC{V}_{R_j ^1} ^\Phi \MC{U}_Q ^\Psi}
 \leq |Q|^{\frac{\alpha}{d} + \frac1q - \frac1p} \OP{\MC{V}_{R_j ^1} ^\Phi
     \MC{U}_Q ^\Psi}
 \leq C' |Q| ^{\frac{\alpha}{d} + \frac1q -\frac1p}
 \|V^{-\frac1p} \MC{U}_Q ^\Psi \|_{\Phi, R_j^1},
\]
where $C' > 1$ depends only on $n$.  Therefore, by the definition of
the Luxemburg norm,
\begin{equation*}
\dashint_{R_j ^1} \Phi
  \bigg(\frac{|V(y)^{-\frac1p} \MC{U}_Q ^\Psi|_\op}{(C'|Q|^{
        \frac{\alpha}{d} + \frac1q - \frac1p})^{-1} C} \bigg) \, dy
>  1.
\end{equation*}

Now set $C = 2 C' \|(U, V)\|' $, where by~\eqref{eqn:reducing-norm2},
\eqref{eqn:mainint-alt} and our assumption on the weights $U$ and $V$,
 \begin{equation*}
  \|(U, V)\|'
 = \sup_Q|Q|^{\frac{\al}{d}+\frac1q-\frac1p} \|V^{-\frac1p}
\MC{U}_Q^\Psi\|_{\Phi, Q}
\lesssim [U,V]_{p,q,\Phi,\Psi}< \infty.
\end{equation*}
 Since the cubes
$\{R_j^1\}$ are disjoint and $\Phi$ is convex, we get
\begin{multline*}
\sum_{j}  |R_j ^1|
 \leq \sum_j \int_{R_j ^1} \Phi \bigg(\frac{|V(y)^{-\frac1p}
  \MC{U}_Q ^\Psi|_\op}{(C'|Q|^{ \frac{\alpha}{d} + \frac1q -
  \frac1p})^{-1}  C} \bigg) \, dy  \\
 \leq \frac{|Q|}{2} \dashint_Q \Phi  \bigg(\frac{|V(y)^{-\frac1p}
\MC{U}_Q ^\Psi|_\op}{  \|V^{-\frac1p}
\MC{U}_Q^\Psi\|_{\Phi, Q} } \bigg) \, dy
 \leq  \frac{|Q|}{2};
\end{multline*}
this proves \eqref{DecayIneq}.

\medskip

To complete the proof we will use an approximation argument.  For
$m\in \mathbb{N}$ such that $2^{-m} < \ell(Q)$, define the truncated operator
 $$N^m_Q(x)=\sup_{\substack{R\in \D(Q)\\x\in R\\
\ell(R)>2^{-m}}}|R|^{\frac{\al}{d}+\frac1q-\frac1p}|U(x)^{\frac1q}\mathcal{V}_R^\Phi|_{\text{op}}.$$
We will prove that
\begin{equation} \label{eqn:truncation}
\dashint_Q \Psi\left(\frac{N_Q^m (x)}{ C C''} \right) \, dx \leq 3 ,
\end{equation}
where
\begin{equation*}
C'' = \sup_Q \|U^\frac1q (\MC{U}_Q ^\Psi)^{-1} \|_{\Psi, Q}
\lesssim 1.
\end{equation*}
(The last inequality follows from~\eqref{eqn:reducing-norm2}.) Then by
convexity and the definition of the Luxemburg norm we will have that
$\|N_Q^m\|_{\Psi, Q} \leq 3 C C''$, and the desired inequality follows
from Fatou's lemma as $m\rightarrow \infty$.

To prove~\eqref{eqn:truncation}, let $G_Q=\bigcup_jR^1_j$.
If $x\in Q\backslash G_Q$, then for any dyadic cube $R\in \D(Q)$
containing $x$ such that $\ell(R)>2^{-m}$ we have
\begin{multline*}
|R|^{\frac{\al}{d}+\frac1q-\frac1p}|U(x)^{\frac1q}\mathcal{V}_R^\Phi|_{\text{op}}
 =|R|^{\frac{\al}{d}+\frac1q-\frac1p}|U(x)^{\frac1q}(\MC{U}_Q
 ^\Psi)^{-1}\mathcal{U}_Q^\Psi \mathcal{V}_R^\Phi|_{\text{op}}\\
 \leq
 |R|^{\frac{\al}{d}+\frac1q-\frac1p}|U(x)^{\frac1q}(\mathcal{U}_Q^\Psi)^{-1}|_{\text{op}}|
\mathcal{U}_Q^\Psi\mathcal{V}_R^\Phi|_{\text{op}}
 \leq C|U(x)^{\frac1q}(\MC{U}_Q ^\Psi)^{-1}|_{\text{op}}.
\end{multline*}
Let $F_j = \{x \in R_j^1 : N_{Q } ^m (x) \neq N_{R_j^1} ^m (x)\}$.
Then by the maximality of the cubes $\{R_j^1\}$ and the previous
estimate, we have that if $x\in F_j$,
$N_{Q} ^m (x) \leq C|U(x)^{\frac1q}(\MC{U}_Q ^\Psi)^{-1}|_{\text{op}}$.

We can now estimate as follows:
\begin{align*}
& \int_Q \Psi\left( \frac{N_Q ^m(x)}{CC''} \right) \, dx \\
& \qquad  \qquad \leq
 \int_{Q \backslash G_Q} \Psi \bigg(\frac{|U^\frac1q (x) (\MC{U}_Q
  ^\Psi)^{-1} |_\op}{C''}\bigg) \, dx
 +
  \sum_{j} \int_{F_j } \Psi \bigg(\frac{|U^\frac1q (x) (\MC{U}_Q
  ^\Psi)^{-1} |_\op}{C''}\bigg) \, dx \\
& \qquad \qquad \qquad \qquad +
  \sum_{j} \int_{R_j^1 \backslash F_j} \Psi \left(\frac{N_Q
  ^m(x)}{CC''} \right) \, dx \\
& \qquad \qquad \leq
2  |Q| +  \sum_{j} \int_{R_j^1} \Psi \left(\frac{N_{R_j^1} ^m(x)}{C C''}
  \right) \, dx.
\end{align*}

To estimate the last term we iterate this argument.  For each $j$ form
the collection $\{R^2_k\}$ of maximal dyadic cubes, if any, $R\in\D(R^1_j)$ such
that
$$|R|^{\frac{\al}{d}+\frac1q-\frac1p}|\mathcal{U}^\Psi_{R^1_j}\mathcal{V}_R^\Phi|_{\text{op}}>C.$$
Then we can repeat the first  argument above to show that for each $j$,
\begin{equation} \label{eqn:sparse}
\sum_{k :R^2_k\subset R_j^1} |R_k^2|\leq \frac{1}{2}|R_j^1|.
\end{equation}
Thus, repeating the second argument we get
\begin{align*}
\sum_j \int_{R_j^1} \Psi \bigg(\frac{N_{R_j^1} ^m(x)}{C C''}  \bigg) \, dx
& \leq \sum_j \sum_{k :R^2_k\subset R_j^1} |R_k^2| + \int_{R_k^w} \Psi
  \bigg(\frac{N_{R_k^2} ^m(x)}{C C''}  \bigg) \, dx \\
& \leq \frac{1}{2}|Q|+ \sum_j \sum_{k :R^2_k\subset R_j^1} \int_{R_k^2} \Psi
  \bigg(\frac{N_{R_k^2} ^m(x)}{C C''}  \bigg) \, dx.
\end{align*}

We continue with this argument on each integral on the right-hand
side.   However, by \eqref{eqn:sparse}, the cubes $R_k^2$ are
properly contained in the cubes $R_j^1$.  But for this argument we are
assuming that all the cubes have side length greater than $2^{-m}$.
Therefore, after $k$ iterations, where $k\geq m+\log_2(\ell(Q))$, the
resulting collection of cubes $\{R^k_i\}$ must be empty so the final
sum in the estimate vanishes.  So if we sum over the $k$ steps, we get
\[ \dashint_Q \Psi\left(\frac{N_Q^m (x)}{ C C''} \right) \, dx
\leq 3-2^{-k} \leq 3.  \]
This gives us~\eqref{eqn:truncation} and our proof is complete.
\end{proof}

\bigskip

\begin{proof}[Proof of Theorem \ref{thm:mainint}]
We will prove that $U ^\frac{1}{q} I_\alpha V^{-\frac1p} :
L^p(\R^d,\C^n) \rightarrow L^p(\R^d,\C^n)$.  By a standard
approximation argument, it will suffice to prove that
\[ \left|\ip{U ^\frac{1}{q} I_\alpha V^{-\frac1p} {f}}{{g}}_{L^2}
  \right|
\lesssim \|f\|_{L^p} \|g\|_{L^{q'}}, \]
where $f,\,g$ are bounded functions of compact support.
In~\cite[Lemma~3.8]{Isralowitz:2016we} it was shown that
\begin{equation*}
\left|\ip{U ^\frac{1}{q} I_\alpha V^{-\frac1p} {f}}{{g}}_{L^2} \right|
\lesssim \sum_{t \in \{0, \pm\frac13\}^d} \sum_{Q \in \D^t} {|Q|
  ^{\frac{\alpha}{d}}} \avgint_Q \int_Q \left|\ip{  V(y)
    ^{-\frac{1}{p}} {f}(y)}{U (x) ^\frac{1}{q} {g}(x) } _{\Cn} \right|
\, dx \, dy,
\end{equation*}
where the dyadic grids $\D^t$ are defined as in
Proposition~\ref{dyadic}.  Therefore, to complete the proof, it
suffices to fix a dyadic grid $\D$ and show that the inner sum is
bounded by $\|f\|_{L^p} \|g\|_{L^{q'}}$.  Our argument adapts to the
matrix setting the scalar, two weight argument originally due to
P\'erez~\cite{perez94} (see also~\cite{CruzUribe:2016ji}).

First note that by the generalized H\"older's
inequality in the scale of Orlicz spaces, inequality~\eqref{eqn:reducing-norm2} and the definition of~$\Vv_Q^\Phi$,
\begin{align*} \sum_{Q \in \D} &{|Q| ^{
      \frac{\alpha}{d}}} \avgint_Q \int_Q \left|\ip{ V(y) ^{-\frac{1}{p}}
      {f}(y)}{U (x) ^\frac{1}{q} {g}(x) } _{\Cn} \right| \, dx \, dy
  \\ & \leq \sum_{Q \in \D} |Q| ^{ \frac{\alpha}{d}}
  \left(\avgint_Q | (\MC{V}_Q ^\Phi)^{-1} V(y)
    ^{-\frac{1}{p}} {f}(y)| \, dy \right)\left( \int_Q | \MC{V}_Q
    ^\Phi U(x) ^\frac{1}{q} {g}(x)| \, dx \right) \\ & \leq \sum_{Q
    \in \D} |Q| ^{ \frac{\alpha}{d}} \| (\MC{V}_Q ^\Phi)^{-1} V
  ^{-\frac{1}{p}}\|_{\Phi, Q} \| {f}\|_{\bar{\Phi}, Q} \left(
    \int_Q | \MC{V}_Q ^\Phi U(x) ^\frac{1}{q} {g}(x)| \, dx \right) \\
  & \leq \sum_{Q \in \D} |Q| ^{ \frac{\alpha}{d}} \|
  {f}\|_{\bar{\Phi}, Q} \left( \int_Q | \MC{V}_Q ^\Phi U(x)
    ^\frac{1}{q} {g}(x)| \, dx \right).
\end{align*}

Fix $a>2^{d+1}$ and define the collection of cubes
 \begin{equation*}
\MC{Q} ^k =
  \{Q \in \D : a^k < \|{f}\|_{\bar{\Phi}, Q} \leq a^{k + 1}\},
\end{equation*}
and let $\MC{S}^k$ be the disjoint collection of
$Q\in \D$ that are maximal with respect to the
inequality $\|{f}\|_{\bar{\Phi}, Q} > a^k$.
Set $\MC{S} = \bigcup_k \MC{S}^k$.
We now continue the above estimate:
\begin{align*}
&  \sum_k \sum_{Q \in \MC{Q}^k} |Q| ^{ \frac{\alpha}{d}}\|
  {f}\|_{\bar{\Phi}, Q}
\left( \int_Q | \MC{V}_Q ^\Phi  U(x) ^\frac{1}{q}  {g}(x)| \, dx
  \right)  \\
& \qquad \qquad \leq  \sum_k  a^{k + 1} \sum_{Q \in \MC{Q}^k} |Q| ^{
  \frac{\alpha}{d}} \int_Q |\MC{V}_Q ^\Phi  U(x) ^\frac{1}{q}  {g}(x)|
  \, dx \\
& \qquad \qquad = \sum_k  a^{k + 1} \sum_{P \in \MC{S}^k} \sum_{\substack{Q \in
  \MC{Q}^k \\ Q \subset P}} |Q| ^{ \frac{\alpha}{d}}
\int_Q |\MC{V}_Q ^\Phi  U(x) ^\frac{1}{q} {g}(x)| \, dx.
\end{align*}

Fix a cube $P\in \Ss^k$; then we can estimate the inner most sum:
\begin{align*}
& \sum_{\substack{Q \in \MC{Q}^k \\ Q \subset P}} |Q| ^{ \frac{\alpha}{d}} \int_Q |\MC{V}_Q ^\Phi U(x) ^\frac{1}{q} {g}(x)| \, dx\\
 & \qquad\qquad  \leq \sum_{\substack{Q \in \D \\ Q \subset P}} |Q| ^{
  \frac{\alpha}{d}} \int_Q |\MC{V}_Q ^\Phi U(x) ^\frac{1}{q} {g}(x)|
  \, dx \\
&\qquad\qquad = \sum_{j = 0}^\infty \sum_{\substack{Q \subset P \\
  \ell(Q) = 2^{-j} \ell(P)}} |Q| ^{ \frac{\alpha}{d}} \int_Q |\MC{V}_Q
  ^\Phi  U(x) ^\frac{1}{q} {g}(x)| \, dx \\
&\qquad\qquad = |P|^\frac{\alpha}{d} \sum_{j = 0}^\infty 2^{-j\alpha} \sum_{\substack{Q \subset P \\ \ell(Q) = 2^{-j} \ell(P)}}  \int_Q |\MC{V}_Q ^\Phi  U(x) ^\frac{1}{q} {g}(x)| \, dx
\\
&\qquad\qquad \lesssim |P|^\frac{\beta}{d}  \int_P  {N}_P (x) |
  {g}(x)| \, dx,
\end{align*}
where $\frac{\beta}{d} = \frac1p - \frac1q$, $\beta\geq 0$ by our
hypotheses, and $N_P$ is defined by~\eqref{eqn:NQ-defn}.

If we insert this estimate into the above inequality, then by the
generalized H\"{o}lder inequality~\eqref{eqn:gen-holder} and Lemma
\ref{NQOrliczLem},
  \begin{align*}
&  \sum_k a^{k + 1} \sum_{P \in \MC{S}^k} |P|^\frac{\beta}{d}  \int_P
  {N}_P (x) |   {g}(x)| \, dx  \\
& \qquad \qquad \leq a \sum_k  \sum_{P \in \MC{S}^k} |P|  (
  |P|^\frac{\beta}{d} \|{f}\|_{\bar{\Phi}, P} ) \left(
  \avgint_{P} {N}_P (x) |   {g}(x)| \, dx\right) \\
& \qquad \qquad\leq a \sum_k  \sum_{P \in \MC{S}^k} |P|  (
  |P|^\frac{\beta}{d} \|{f}\|_{\bar{\Phi}, P} ) (\| {N}_P
  \|_{\Psi, P}     \|{g}\|_{\bar{\Psi}, P})  \\
&\qquad \qquad\leq  a \sum_k  \sum_{P \in \MC{S}^k} |P|  \, \inf_{x \in P}
  M^\beta _{\bar{\Phi}}{f}(x)  M_{\bar{\Psi}} {g} (x).
  \end{align*}

For each  $Q \in \MC{S}$, define
 \begin{equation*}
E_Q = Q \backslash \bigcup_{\substack{Q' \in \MC{S} \\ Q' \subsetneq
    Q}} Q'.
\end{equation*}
Then by~\cite[Proposition~A.1]{MR2797562}, the sets $E_Q$ are pairwise
disjoint and $|E_Q|\geq \frac{1}{2}|Q|$.    Given this, we can
continue the above estimate:
\begin{align*}
\sum_k  \sum_{P \in \MC{S}^k} &  |P|  \, \inf_{x \in P}M^\beta _{\bar{\Phi}}{f}(x)  M_{\bar{\Psi}} {g} (x)   \\
& \leq 2   \sum_{Q \in \MC{S}} |E_Q|  \, \inf_{x \in Q}M^\beta
  _{\bar{\Phi}}{f}(x)  M_{\bar{\Psi}} {g} (x)    \\
 & \leq 2   \sum_{Q \in \MC{S}} \int_{E_Q} M^\beta _{\bar{\Phi}}f
   \, M_{\bar{\Psi}}g \, dx
    \\ & \leq 2 \int_{\mathbb{R}^d} M^\beta _{\bar{\Phi}}  {f} \, M_{\bar{\Psi}} {g} \, dx
    \\ & \leq 2  \|M_{\bar{\Phi}} ^\beta{f}\|_{L^q} \|M_{\bar{\Psi}} {g}\|_{L^{q'}}
     \\ & \lesssim \|{f}\|_{L^{p}} \|{g}\|_{L^{q'}}.
\end{align*}
The last inequality follows from~\eqref{eqn:frac-orliz-max}.  If we
combine  all of the above inequalities, we get the desired
result.
\end{proof}

\section{Proof of Theorem~\ref{thm:CZO} and
  Corollary~\ref{SharpSparseCor}}
\label{section:CZO}

Throughout this section, for brevity we will write
$\Uu_Q^\Psi=\Uu_Q^{p,\Psi}$ and $\Vv_Q^\Phi=\Vv_Q^{p,\Phi}$.
In order to prove our results about Calder\'on-Zygmund operators we
introduce the concept of sparse operators.  For complete details,
see~\cite{CruzUribe:2016ji}.
Given a dyadic grid $\Dd$,  a set $\Ss \subset \D$ is sparse if for
each cube $Q\in \Ss$, there exists a set $E_Q\subset Q$ such that
$|E_Q| \geq \frac{1}{2}|Q|$ and the collection of sets $\{E_Q\}$ is
pairwise disjoint.    Define the dyadic sparse operator $T^\Ss_\alpha$ by
\[ T^\Ss_\alpha f(x) = \sum_{Q\in \Ss} |Q|^{\frac{\alpha}{d}}\avgint_Q f(y)\,dy \cdot \chi_Q(x). \]
Note that in the proof of Theorem~\ref{thm:mainint} the set of cubes
$\Ss$ is sparse, and the sums being approximated can be viewed as the
integrals of sparse operators.  By modifying this proof we can
prove the following result.

\begin{theorem} \label{thm:sparse}
Given $0\leq \alpha <d$ and  $1<p\leq q<\infty$ such that
$\frac{1}{p}-\frac{1}{q}\leq \frac{\alpha}{d}$, suppose that $\Phi$ and
  $\Psi$ are Young functions with $\bar{\Phi}\in B_{p,q}$ and
  $\bar{\Psi}\in B_{q'}$.  If $(U,V)$ is a pair of matrix weights
  satisfy the bump condition~\eqref{matrixbump2}, then
  $T_\alpha^\Ss : L^p(V) \rightarrow L^q(U)$.
\end{theorem}

\begin{remark}
In the one weight case, a quantitative version of
Theorem~\ref{thm:sparse} was proved
in~\cite{Bickel:2015ub,Isralowitz:2015uj} when $p=q=2$ and
$\alpha=0$.
\end{remark}

\begin{proof}
The proof is virtually identical to the proof of
Theorem~\ref{thm:mainint} above, except that, since we start an
operator defined over a sparse family $\Ss$, we may omit the argument
used to construct the set $\Ss$.  This was the only part of the proof
of Theorem~\ref{thm:mainint} where we used the assumption that
$\alpha>0$; everywhere else in the proof we may take $\alpha=0$.

Because of these similarities, we only sketch the main steps:
\begin{align*}
& \left|\ip{U ^\frac{1}{q} T^\Ss_\alpha V^{-\frac1p} {f}}{{g}}_{L^2}
  \right| \\
& \qquad \qquad \leq \sum_{Q\in\Ss} |Q|^{1+\frac{\alpha}{d}} \avgint_Q \avgint_Q
\bigg|\bigg\langle
  V(y)^{-\frac{1}{p}}(y)f(y),U(x)^{\frac{1}{q}}(x)g(x)\bigg\rangle_{\Cn}\bigg|\,dx\, dy  \\
& \qquad \qquad  \leq \sup_{{Q} }
  |{Q}|^{\frac{\alpha}{d}+\frac{1}{q}-\frac{1}{p}}
|\MC{V}_{{Q}} ^\Phi \MC{U}_{{Q}} ^\Psi|_{\op} \\
& \qquad \qquad \qquad \times
\sum_{Q \in \Ss} |E_Q|  \left(|Q| ^ \frac{\beta}{d} \dashint_Q
  |(\MC{V}_Q ^\Phi) ^{-1} V^{-\frac{1}{p}} {f}|\,dx \right)
  \left(\dashint_Q |(\MC{U}_Q ^\Psi)^{-1} U^\frac{1}{q} {g}| \,dx \right)\\
&  \qquad \qquad
\lesssim \|M_{\bar{\Phi}} ^\beta {f}\|_{L^{q}} \|M_{\bar{\Psi}} {g}
  \|_{L^{q'}}. \\
&  \qquad \qquad
\lesssim \|{f}\|_{L^p}  \|{g}\|_{L^{q'}}.
\end{align*}
\end{proof}

We will now use Theorem~\ref{thm:sparse} with $p=q$ and $\alpha=0$ (or
more precisely, its proof) to prove Theorem~\ref{thm:CZO} and
Corollary~\ref{SharpSparseCor}.  To do so, we must first describe the
recent results of Nazarov, {\em et al.}~\cite{Nazarov:2017ufa} on
convex body domination.  Fix a cube $Q$ and a $\Cn$ valued function
$f \in L^1(Q)$.  Define
\begin{equation*}
\lla f \rra_Q = \left\{ \dashint_Q \varphi f\,dx  :  \varphi : Q \rightarrow \R, \, \|\varphi\|_{L^\infty(Q)}
    \leq 1\right\};
\end{equation*}
Then $\lla f \rra_Q$ is a symmetric, convex, compact set in $\C^n$.
If $T$ is a CZO (or a Haar shift or a
paraproduct) then for $f \in L^1(Q)$, $Tf$ is dominated by a
sparse convex body operator.  More precisely,  there exists a sparse collection
$\Ss$ such that for some constant $C$ independent of $f$, and
a.e. $x\in \Rd$,
\begin{equation}  \label{eqn:convex-domination}
Tf (x) \in C\sum_{Q \in \Ss} \lla f
  \rra_Q \chi_Q (x),
\end{equation}
where the sum is an infinite Minkowski sum of convex bodies.

As a consequence of this fact, to prove norm inequalities for a CZO,
it is enough to prove uniform estimates for the generalized sparse
operators of the form
\[ T^\Ss f(x) = \sum_{Q\in \Ss}  \avgint_Q
  \varphi_Q (x, y)f(y)\,dy, \]
where for each $Q$, $\varphi_Q$ is a real valued function
supported on $Q$ as a function of $y$ and such that for each $x$,  $\|\varphi_Q(x, \cdot)\|_\infty \leq 1$.  Note that it is not clear from ~\cite{Nazarov:2017ufa} whether $\varphi_Q(x, y)$ can be chosen as a measurable function of $x$, though this is not important for us (and is unlikely to be important for the further study of matrix weighted norm inequalities.)  

\begin{proof}[Proof of Theorem~\ref{thm:CZO}]
Since \begin{align*} \bigg| &  \ip{U^\frac{1}{p} T V^{-\frac{1}{p}} f (x)}{g(x)} _{\Cn}\bigg| \\  & = \left|\sum_{Q \in \Ss}  \chi_Q (x)  \ip{ \MC{U}_{{Q}} ^\Psi  \MC{V}_{{Q}} ^\Phi  \dashint_Q  \varphi_Q (x, \cdot)  (\MC{V}_Q ^\Phi)^{-1} V^{-\frac{1}{p}}  f }{(\MC{U}_{{Q}} ^\Psi )^{-1} U^\frac{1}{p} (x) g(x)}_{\Cn}   \right| \\ & \lesssim
 \sup_{\W{Q} }  |\MC{V}_{\W{Q}} ^{\Phi}  \MC{U}_{\W{Q}} ^\Psi|_{\op} \sum_{Q \in \Ss} \left(\dashint_Q  | (\MC{V}_Q ^p )^{-1} V^\frac{1}{p} (x)   f| \right)  \left(\chi_Q (x) |(\MC{U}_{{Q}} ^\Psi)^{-1} U^{-\frac{1}{p}}g(x)|\right) \end{align*} the proof now continues exactly as in the proof of Theorem~\ref{thm:sparse}.

\end{proof}

To prove Corollary~\ref{SharpSparseCor} we first need a few additional
facts about scalar weights and Orlicz maximal operators due to Hyt\"onen and
P\'erez.    We say that a weight
$w\in A_\infty$ if it satisfies the Fujii-Wilson condition
\[ [w]_{A_\infty} = \sup_Q \frac{1}{w(Q)}\int_Q M(w\chi_Q)(x)\,dx <
  \infty. \]
(There are several other definitions of the $A_\infty$ condition:
see~\cite{MR3473651}.  This definition, which seems to yield the
smallest constant, has proved to be the right choice in the study of
sharp constant inequalities for CZOs.)  In~\cite{MR3092729} they
showed that if $w\in A_\infty$, then it satisfies a sharp reverse
H\"older inequality: for any cube $Q$, $w\in RH_s$:  i.e.,
\[ \left(\avgint_Q w^s\,dx\right)^{\frac{1}{s}}
\leq 2\avgint_Q w\,dx, \]
where $s= 1+ \frac{1}{2^{d+11}[w]_{A_\infty}}$.

They also proved a
quantitative version of inequality~\eqref{eqn:frac-orliz-max}:
in~\cite{MR3327006} they showed that given a Young function $\Phi$,
\[ \|M_{\bar{\Phi}}\|_{L^p}
\leq c(n) \left(\int_{\bar{\Phi}(1)}^\infty
\left(\frac{t}{\Phi(t)}\right)^p \,d\Phi(t) \right)^{\frac{1}{p}}. \]
In particular, if we let $\Phi(t)=t^{rp'}$, $r>1$, then a
straightforward computation shows that
\begin{equation} \label{eqn:sharp-max}
 \|M_{\bar{\Phi}}\|_{L^p} \lesssim (r')^{\frac{1}{p}}.
\end{equation}

\begin{proof}[Proof of Corollary~\ref{SharpSparseCor}]
By the argument in the proof of Theorem~\ref{thm:CZO} it is enough to
prove this estimate for sparse operators.    Fix a dyadic grid $\D$
and a sparse set $\Ss\subset \D$ and let $W$ be a matrix $A_p$
weight. As we noted in the introduction, for every $e\in \C^n$,
$|W^{\frac{1}{p}}e|$ is a scalar $A_p$ weight with uniformly bounded
constant~\cite[Corollary~2.2]{G}.  Using the Fujii-Wilson condition, we define
\[ [W]_{A_{p,\infty}^{\sca} }
 = \sup_{e \in \Cn} [|W^{\frac{1}{p}}e|^p ]_{A_\infty}. \]
By the sharp reverse H\"older inequality, if we let
\begin{equation} \label{eqn:rh-exp}
s = 1 + \frac{1}{2^{d+11} [W]_{A_{p,\infty}^{\sca}}}, \qquad
r =  1 + \frac{1}{2^{d+11}
  [W^{-\frac{p'}{p}}]_{A_{p',\infty}^{\sca}}},
\end{equation}
then for every $e\in \C^n$ , $|W^{\frac{1}{p}}e|^p\in RH_s$ and
$|W^{-\frac{1}{p}}e|^{p'}\in RH_r$.

Define $\Psi(t)=t^{sp}$ and $\Phi(t)=t^{rp'}$.   Then $\bar{\Psi}\in
B_{p'}$ and $\bar{\Phi}\in B_p$.  Moreover, we claim that
\[  [W,W]_{p,\Phi,\Psi} \lesssim [W]_{A_p}^{\frac{1}{p}}.  \]
To see this, we argue as in the proof of
Proposition~\ref{prop:reducing-norm}.  Let $\{e_j\}_{j=1}^n$ be an
orthonormal basis in $\C^n$.  Then by~\eqref{eqn:CZO-alt} (with $p=q$
and $U=V=W$), and the reverse H\"older inequality,
\begin{align*}
[W,W]_{p,\Phi,\Psi}
& \approx \sup_Q |\Uu_Q^\Psi\Vv_Q^\Phi|_\op \\
& \approx \sup_Q \sum_{j=1}^n
\left(\avgint_Q |W^{-\frac{1}{p}}(x)\Uu_Q^\Psi
  e_j|^{rp'}\,dx\right)^{\frac{1}{rp'}} \\
& \leq 2\sup_Q \sum_{j=1}^n
\left(\avgint_Q |W^{-\frac{1}{p}}(x)\Uu_Q^\Psi
  e_j|^{p'}\,dx\right)^{\frac{1}{p'}} \\
& \lesssim \sup_Q |\Uu_Q^\Psi \Vv_Q^{p'}|_\op.
\end{align*}
If we repeat this argument again, exchanging the roles of $\Uu$ and
$\Vv$, we get that
\begin{equation} \label{eqn:Ap-cond}
 [W,W]_{p,\Phi,\Psi} \lesssim \sup_Q |\Uu_Q^\Psi \Vv_Q^{p'}|_\op
\lesssim \sup_Q |\Uu_Q^p\Vv_Q^{p'}|_\op
\lesssim [W]_{A_p}^{\frac{1}{p}}.
\end{equation}

Therefore, we can apply Theorem~\ref{thm:CZO} with the pair of weights
$(W,W)$.  A close examination of the proof of this result (i.e., the
proof of Theorem~\ref{thm:sparse}) shows that
\[ \|T\|_{L^p(W)} \lesssim [W,W]_{p,\Phi,\Psi}
  \|M_{\bar{\Phi}}\|_{L^p} \|M_{\bar{\Psi}}\|_{L^{p'}}. \]
But by~\eqref{eqn:Ap-cond} and by~\eqref{eqn:sharp-max} combined
with~\eqref{eqn:rh-exp} we get
\[ [W,W]_{p,\Phi,\Psi}
  \|M_{\bar{\Phi}}\|_{L^p} \|M_{\bar{\Psi}}\|_{L^{p'}}
\lesssim [W]_{A_p}^{\frac{1}{p}}
 [W]_{A_{p,\infty}^{\sca}}^{\frac{1}{p'}}
 [W^{-\frac{p'}{p}}]_{A_{p',\infty}^{\sca}}^{\frac{1}{p}}. \]
This gives us the first estimate in Corollary~\ref{SharpSparseCor};
the second follows from this one, Lemma~\ref{lemma:matrix-dual} and
the fact that
\[ [W]_{A_{p,\infty}^{\sca}} \leq \sup_{e\in \C^n}[|W^{\frac{1}{p}}e|^p]_{A_p} \leq
  [W]_{A_p};
\]
see~~\cite[Corollary~2.2]{G}.
\end{proof}

 \begin{remark}
In~\cite{Nazarov:2017ufa} they proved that the sparse matrix
domination inequality~\eqref{eqn:convex-domination} holds if $T$ is a
Haar shift or a paraproduct.  Consequently, Theorem~\ref{thm:CZO} and
Corollary~\ref{SharpSparseCor} hold for these operators.
Additionally, they proved a slightly stronger result when $p=2$,
assuming that a pair of matrix weights $[U,V]$ satisfy the two weight
$A_p$ condition, and each of $U$ and $V$ satisfy the appropriate
scalar $A_\infty$ condition.  We can immediately extend our proofs to
give the analog of this result for all $1<p<\infty$.  Details are left
to the interested reader.
\end{remark}

\section{Proof of Theorems~\ref{theorem:avg-op}
  and~\ref{thm:mainweak}}
\label{section:characterization}

For brevity, in this section if $\alpha=0$ we will write $A_{p,q}=A_{p,q}^0$; if $p=q$ we will write $A_{p}^\alpha$ or $A_p$ if $\alpha=0$.

\begin{proof}[Proof of Theorem~\ref{theorem:avg-op}]
  We first prove the sufficiency of the $A_{p,q}^\alpha$ condition.
  When $p>1$   we estimate  using
  H\"older's inequality and~\eqref{eqn:Apq-alt}:
\begin{align*}
& \|A_\Q f\|_{L^q(U)}^q  \\
& \qquad = \int_{\Rd} |U^{\frac{1}{q}}(x)A_\Q f(x)|^q\,dx \\
& \qquad  = \int_{\Rd} \bigg| \sum_{Q\in \Q}
|Q|^{\frac{\alpha}{d}}\avgint_Q \chi_Q(x) U^{\frac{1}{q}}(x)
V^{-\frac{1}{p}}(y) V^{\frac{1}{p}}(y)f(y)\,dy\bigg|^q\,dx \\
& \qquad  \leq \int_{\Rd} \sum _{Q\in \Q} |Q|^{\frac{q\alpha}{d}}\chi_Q(x)
\bigg(\avgint_Q |U^{\frac{1}{q}}(x) V^{-\frac{1}{p}}(y)|_\op^{p'}\,dy\bigg)^{\frac{q}{p'}}
\bigg(\avgint_Q |V^{\frac{1}{p}}(y)f(y)|^p\,dy\bigg) ^{\frac{q}{p}}\,dx \\
& \qquad  = \sum _{Q\in \Q} |Q|^{q\frac{\alpha}{d}+1-\frac{q}{p}}\avgint_Q \bigg(\avgint_Q |U^{\frac{1}{q}}(x) V^{-\frac{1}{p}}(y)|_\op^{p'}\,dy\bigg)^{\frac{q}{p'}}\,dx
\bigg(\int_Q |V^{\frac{1}{p}}(y)f(y)|^p\,dy\bigg) ^{\frac{q}{p}} \\
& \qquad  \lesssim [U,V]_{A_{p,q}^\alpha}^q \bigg(\sum _{Q\in \Q} \int_Q |V^{\frac{1}{p}}(y)f(y)|^p\,dy\bigg) ^{\frac{q}{p}} \\
& \qquad \leq [U,V]_{A_{p,q}^\alpha}^q  \|f\|_{L^p(V)}^q.
\end{align*}

When $p=1$ we can argue as above, except that instead of H\"older's
inequality we use Fubini's theorem and~\eqref{eqn:A1q-alt}.

\medskip

To prove necessity when $p>1$, fix a cube $Q$ and let $e\in \Cn$ be
such that $|e|=1$.  Then, assuming
averaging operators are uniformly bounded with norm at most $K$, we
have by duality that there exists $g\in L^p(V)$, $\|g\|_{L^p(V)}=1$,
such that
\begin{align*}
|\Vv^{p'}_Q\Uu^q_Qe|
& \approx \left(\avgint_Q |V^{-\frac{1}{p}}(y) \Uu^q_Qe|^{p'}\,dy\right)^{\frac{1}{p'}} \\
& = |Q|^{- \frac{1}{p'}}\|\chi_Q\Uu^q_Qe\|_{L^{p'}(V^{-\frac{p'}{p}})} \\
& = |Q|^{- \frac{1}{p'}}\int_Q \ip{\Uu^q_Qe}{g(x)}_{\Cn} \,dx \\
& = |Q|^{\frac{1}{p}}\ip{e}{\Uu_Q^q\avgint_Q g(x) \,dx }_{\Cn} \\
& \leq |Q|^{\frac{1}{p}}\left|\Uu_Q^q\avgint_Q g(x)\,dx \right|\\
& \approx |Q|^{\frac{1}{p}} \left(\avgint_Q |U^{\frac{1}{q}}(y)A_Qg(y)|^q\,dy\right)^{\frac{1}{q}} \\
& = |Q|^{\frac{1}{p}-\frac{1}{q}-\frac{\alpha}{d}} \|A_Q^\alpha g\|_{L^q(U)} \\
& \leq K |Q|^{\frac{1}{p}-\frac{1}{q}-\frac{\alpha}{d}} \|g\|_{L^p(V)}\\
& \leq K |Q|^{\frac{1}{p}-\frac{1}{q}-\frac{\alpha}{d}}.
\end{align*}
If we now rearrange terms and take the supremum over all $Q$ we get that
\[ \sup_Q |Q|^{\frac{\alpha}{d}+\frac{1}{q}-\frac{1}{p}}|\Uu_Q^q\Vv_Q^{p'}|_\op
 = \sup |Q|^{\frac{\alpha}{d}+\frac{1}{q}-\frac{1}{p}}|\Vv_Q^{p'}\Uu_Q^q|_\op \lesssim K, \]
and so $(U,V)\in A_{p,q}^\alpha$.

When $p=1$ we cannot use duality, so we argue as follows.  Since $A_Q$
is linear,  given $f \in L^1(Q)$ we can rewrite our assumption to get
 \begin{equation*}
\|A_Q f\|_{L^q(Q)} =
\bigg(\int_Q \left| |Q|^{\frac{\alpha}{d}} \dashint_Q
    U^{\frac{1}{q}} (x) V^{-1} (y) {f}(y) \, dy \right|^q \, dx
\bigg)^{\frac{1}{q}}
\leq C   \|{f}\|_{L^1 (Q)}.
\end{equation*}
Therefore, given any  $S \subseteq Q$ with
$|S| > 0$, if we let ${f} (x) = \chi_S (x){e} $, where $e\in \C^n$ and
$|e|=1$, then
 \begin{equation*}
|S| |Q| ^{\frac{1}{q} - 1 +  \frac{\alpha}{d}}
\bigg(\dashint_Q \left|U^\frac{1}{q} (x)
    \left(\dashint_S V^{-1}(y) {e}\,dy \right) \right|^q \,
  dx\bigg)^{\frac{1}{q}}
 \leq   K|S|.
\end{equation*}
Thus, by the definition of $\MC{V}_Q ^q$ we
get that
\begin{equation*}
|Q| ^{\frac{1}{q} - 1 +  \frac{\alpha}{d}}
\left| \MC{U}_Q ^q \left(\dashint_S  V^{-1}(y) {e}\,dy\right)  \right|
\lesssim K.
\end{equation*}
But then by the Lebesgue differentiation theorem it follows that
\begin{equation*}
|Q| ^{\frac{\alpha}{d} + \frac{1}{q} -  1 } \esssup_{y\in Q} \left|
  \MC{U}_Q ^q  V^{-1} (y)  \right|_{\op}
 \lesssim K.
\end{equation*}
By~\eqref{eqn:reducing-norm2} it follows that $(U,V)\in
A_{p,q}^\alpha$.
\end{proof}

\medskip

As a corollary to Theorem~\ref{theorem:avg-op} we have the uniform
boundedness of convolution operators and the convergence of
approximate identities.

\begin{corollary} \label{cor:convolution-op}
Given $1\leq p<\infty$ and a pair of matrix weights $(U,V)$ in $A_p$, let $\varphi \in C_c^\infty(B(0,1))$ be a non-negative, radially symmetric and decreasing function with $\|\varphi\|_1=1$, and for $t>0$ let $\varphi_t(x)=t^{-n}\varphi(x/t)$.  Then
\[ \sup_{t>0} \|\varphi_t * f\|_{L^p(U)}
\leq C\|f\|_{L^p(V)}. \]
Moreover, we have that
\[ \lim_{t\rightarrow 0} \|\varphi_t*f -f\|_{L^p(U)} = 0. \]
\end{corollary}

This was proved in the one weight case in~\cite[Theorem~4.9]{CRM}.
The proof is essentially the same, bounding the convolution operator
by averaging operators and then applying Theorem~\ref{theorem:avg-op}.
Details are left to the interested reader, except for the following
result which is of independent interest.

Recall that if $(u, v)$ is a pair of scalar $A_p$weights, then it
is immediate by the Lebesgue differentiation theorem that
$u(x) \leq [u, v]_{A_p} v(x)$ a.e.    The following result is the
matrix analog.

\begin{proposition} \label{PointwiseProp} Given $1\leq p<\infty$, if
  $(U, V) \in \text{A}_p$,  then
\[ |U^{\frac{1}{p}}(x)V^{-\frac{1}{p}}(x)|_\op
\lesssim [U,  V]_{\text{A}_p} ^\frac{1}{p}.  \]
 \end{proposition}

\begin{remark}
In the proof of Corollary~\ref{cor:convolution-op}, this is used to
prove that the $L^p(U)$ norm of a function is dominated by the
$L^p(V)$ norm:
\[ \|f\|_{L^p(U)}
\leq \bigg(\int_{\R^d} |U^{\frac{1}{p}}(x)V^{-\frac{1}{p}}(x)|_\op
|V^{\frac{1}{p}}(x)f(x)|\,dx\bigg)^{\frac{1}{p}}
\lesssim  [U,  V]_{\text{A}_p} ^\frac{1}{p} \|f\|_{L^p(V)}. \]
\end{remark}

\medskip

\begin{proof}
We first consider the case when $p>1$.
Since $U$ is locally integrable, we have for a.e. $x \in \Rd$
  that
\begin{equation*}
\lim_{m \rightarrow \infty} \left|U(x) -
      \dashint_{Q_m ^x} U(y)\,dy \right|_{\op}
= \lim_{m \rightarrow \infty}
    \left|U^\frac{1}{p}(x) - \dashint_{Q_m ^x}
      U^\frac{1}{p}(y)\,dy\right|_{\op} = 0,
 \end{equation*}
and that the same holds for $V, V^{-\frac{1}{p'}}, $ and the scalar function
  $|U|_{\op}$; here  $\{Q_m ^{x} \}$ is a sequence of nested cubes whose
  intersection is $\{x\}$ and whose side-length tends to zero.
Thus  by H\"{o}lder's inequality, for any
  ${e} \in \Cn$ we have
\begin{multline*} |
    U^\frac{1}{p} (x) {e}|^p
 = \lim_{m \rightarrow \infty}
    \left|\dashint_{Q_m ^x} U^\frac{1}{p}(y) {e}\,dy\right| ^p \\
\leq
    \limsup_{m \rightarrow \infty} \left(\dashint_{Q_m ^x}
      |U^\frac{1}{p}(y) {e}\,dy|\right) ^p
 \leq \limsup_{m \rightarrow
      \infty} \dashint_{Q_m ^x} |U^\frac{1}{p}(y) {e}|^p\,dy
 \approx
    \limsup_{m \rightarrow \infty} |\MC{U}_{Q_m ^x} ^p {e} |
    ^p.
  \end{multline*}

  On the other hand,
\begin{equation*}
\limsup_{m      \rightarrow \infty} |\MC{U}_{Q_m ^x} ^p |_{\op} ^p
\approx
    \limsup_{m \rightarrow \infty} \sum_{j = 1}^n \dashint_{Q_m ^x}
    |U^\frac{1}{p}(y) {e}_j|^p\,dy
\leq \limsup_{m \rightarrow \infty}
    \dashint_{Q_m ^x} |U(y)|_{\op}\,dy = |U(x)|_{\op};
\end{equation*}
 in
  particular, $\{|\MC{U}_{Q_m ^x} ^p |_{\op}\}$ is bounded.  Then we
  can argue as we did above
  above to get that for any ${e} \in \Cn$,
 \begin{equation*}
    \limsup_{m \rightarrow \infty} |V^{-\frac{1}{p}} (x) \MC{U}_{Q_m^
      x} ^p {e} |^{p'} \lesssim \limsup_{m \rightarrow \infty} |
    \MC{V}_{Q_m ^x} ^{p'} \MC{U} _{Q_m ^x} ^p {e}|^{p'}
    . \end{equation*}

 Hence, we get that
\begin{align*}
|U^\frac{1}{p} (x) V^{-\frac{1}{p}} (x)|_{\op}
& \approx \left( \sum_{j = 1}^n |U^\frac{1}{p} (x) V^{-\frac{1}{p}}
  (x){e}_j | ^p \right)^\frac{1}{p}  \\
& \lesssim \limsup_{m \rightarrow \infty} \left(\sum_{j = 1}^n |
  \MC{U}_{Q_m^ x} ^p  V^{-\frac{1}{p}} (x) {e}_j|^p\right)^\frac{1}{p}
  \\
& \approx  \limsup_{m \rightarrow \infty} |V^{-\frac{1}{p}} (x)
  \MC{U}_{Q_m^ x} ^p  |_{\op} \\
& \approx \limsup_{m \rightarrow \infty} \left(\sum_{j = 1}^n
  |V^{-\frac{1}{p}} (x) \MC{U}_{Q_m^ x} ^p {e}_j |^{p'}
  \right)^\frac{1}{p'}   \\
& \lesssim \limsup_{m \rightarrow \infty}  | \MC{V}_{Q_m ^x} ^{p'}
  \MC{U} _{Q_m ^x} ^p |_{\op} \\
& \lesssim [U, V]_{\text{A}_p} ^\frac{1}{p}.
\end{align*}
\end{proof}

\bigskip

\begin{proof}[Proof of Theorem \ref{thm:mainweak}]
 We first prove $(1)$ implies $(2)$.   Given $(U,V)\in
 A_{p,q}^\alpha$,  we will prove that
  $M_{\alpha,U,V}' : L^p\rightarrow L^{q,\infty}$.  Arguing exactly as
  we did in Section~\ref{section:proof-mainmax} using
  Proposition~\ref{dyadic}, it will suffice to fix a dyadic grid $\D$
  and prove that
  $M_{\alpha, U, V, \D} ' : L^p \rightarrow L^{q, \infty}$, where
  $M_{\alpha, U, V, \D}$ is defined as in \eqref{eqn:two-wt-auxiliary}
  but with the supremum restricted to cubes in $\D$.

Fix $\lambda > 0$ and let ${f} \in L^p(\Rd, \Cn)$.  Then for
  any cube $Q\in \D$ we have by~\eqref{eqn:Apq-alt} that
\[ {|Q| ^{ \frac{\alpha}{d}}} \avgint_{Q} |\MC{U}_{Q} ^q V^{-\frac1p} (y)
    {f} (y) | \, dy
 \lesssim |Q|^{\frac{\alpha}{d} - \frac1p}
    |\MC{U}_Q ^q \MC{V}_Q ^{p'}|_\op \|f\|_{L^p}
 \lesssim |Q|^{-\frac{1}{q}} [U,V]_{A_{p,q}^\alpha} \|f\|_{L^p}.  \]
The right-hand side tends to $0$ as $|Q| \rightarrow \infty$, so
(see~\cite[Proposition~A.7]{MR2797562}) there
exists a collection
  $\{Q_j\}$ of maximal, disjoint cubes in $\D$ such that
 \begin{equation*}
{|Q_j| ^{ \frac{\alpha}{d}}} \avgint_{Q_j} |\MC{U}_{Q_j} ^q V^{-\frac1p} (y) {f} (y) | \, dy >
    \lambda
\end{equation*}
and
 \begin{equation*}
\bigcup_j Q_j    = \{x : M_{\alpha, U, V, \D} ' {f} (x) > \lambda \}.
\end{equation*}
But then we can estimate as follows: by H\"older's inequality and the definition of $\Vv_Q^{p'}$,
\begin{align*}
& |\{x : M_{\alpha, U, V, \D} ' {f} (x) > \lambda \}| \\
& \qquad \qquad = \sum_{j} |Q_j|\\
& \qquad \qquad \leq
\frac{1}{\lambda^q} \sum_j
\left( {|Q_j| ^{ -1+\frac1q +\frac{\alpha}{d}}}
\int_{Q_j} |\MC{U}_{Q_j} ^q V(y) ^{-\frac1p }
  {f}(y)| \, dy \right)^q \\
& \qquad \qquad \leq
\frac{1}{\lambda^q} \sum_j
\left( {|Q_j| ^{ -1+\frac1q +  \frac{\alpha}{d}}} \right)^q
\left(\avgint_{Q_j} |\Uu_{Q_j}^q V(y) ^{-\frac1p  }|^{{p'}}\,dy\right)^{\frac{q}{p'}}
\left(\int_{Q_j}|f(y)|^p\,dy\right)^{\frac{q}{p}} \\
& \qquad \qquad  \lesssim \frac{1}{\lambda^q} \sum_{j}
\left( {|Q_j| ^{ -1+\frac1q +  \frac{\alpha}{d}}} \right)^q
|\MC{U}_{Q_j} ^q \MC{V}_{Q_j} ^{p'} |_\op^q \left(\int_{Q_j}
  |{f}(y)| ^p \, dy\right)^\frac{q}{p};  \\
\intertext{by~\eqref{eqn:Apq-alt},}
& \qquad \qquad \leq [U,V]_{A_{p,q}^\alpha}\lambda^{-q} \sum_j
\left(\int_{Q_j}  |{f}(y)| ^p \, dy\right)^\frac{q}{p} \\
& \qquad \qquad \leq  [U,V]_{A_{p,q}^\alpha}\lambda^{-q}
  \|f\|_{L^p}^q; \end{align*}
the last inequality holds since $q\geq p$ (so by convexity we may pull
the power outside the sum), and since the cubes $\{Q_j\}$ are disjoint.
This completes the proof that $(1)$ implies $(2)$.

\medskip

The proof that $(2)$ implies $(3)$ is immediate: given a cube
$Q$,  $B^\alpha_Q f(x) \leq M_{\alpha,U,V}'f(x)$.

\medskip

Finally, we prove that $(3)$ implies $(1)$.
It follows at once from the definition of the $L^{q,\infty}$ norm that
for any $e\in \C^n$,
$|Q|^{-\frac1q } \|\chi_Q {e} \|_{L^{q, \infty}} = |{e}|$.
First suppose that $p>1$.  Then using this identity, duality,
and~\eqref{eqn:matrix-Apq}, we have that
 \begin{align*}
\sup_Q \sup_{\|{f}\|_{L^p} = 1}
&  \left\| {\chi_Q}{|Q|^{\frac{\alpha}{d}  }} \avgint_Q \MC{U}_Q ^q V^{-\frac1p} (y) {f}(y) \, dy \right\|_{L^{q, \infty}}
\\
& = \sup_Q  \sup_{\|{f}\|_{L^p} = 1} \left| {|Q|^{\frac{\alpha}{d}
  +\frac1q  } }
\avgint_Q \MC{U}_Q ^q V^{-\frac1p} (y)
  {f}(y) \, dy \right| \\
& =  \sup_Q  \sup_{\|{f}\|_{L^p} = 1} \sup_{|{e}| = 1}
{|Q|^{\frac{\alpha}{d} +\frac1q }} \avgint_Q \ip{\MC{U}_Q
  ^q V^{-\frac1p} (y) {f}(y)}{{e}}_{\Cn}  \, dy  \\
& =
\sup_Q  \sup_{|{e}| = 1}  \sup_{\|{f}\|_{L^p} = 1}  {|Q|^{
 \frac{\alpha}{d} + \frac1q}} \avgint_Q \ip{\MC{U}_Q ^q V^{-\frac1p}
  (y) {f}(y)}{{e}}_{\Cn}  \, dy  \\
& = \sup_Q  \sup_{|{e}| = 1}  \sup_{\|{f}\|_{L^p} = 1}
  {|Q|^{\frac{\alpha}{d} +\frac1q } } \avgint_{\Rd} \ip{
  {f}(y)}{\chi_Q V^{-\frac1p} (y)\MC{U}_Q ^q {e}}_{\Cn}  \, dy  \\
& = \sup_Q  \sup_{|{e}| = 1} {|Q|^{\frac{\alpha}{d}
  +\frac1q -1 }} \| \chi_Q V^{-\frac1p }\MC{U}_Q ^q {e}\|_{L^{p'}}  \\
& \approx \sup_Q  {|Q|^{ \frac{\alpha}{d} +\frac1q - \frac1p }}
  |\MC{V}_Q ^{p'} \MC{U}_Q ^p |_\op \\
& \approx [U,V]_{A_{p,q}^\alpha}.
\end{align*}
When $p=1$ the proof is nearly the same, except that instead of using
duality to get the $L^{p'}$ norm, we take the operator norm of the
matrices and use~\eqref{eqn:A1q-alt}.   This completes the proof that
$(3)$ implies $(1)$.
\end{proof}

\section{Proof of Theorem~\ref{thm:poincare}}
\label{section:poincare}

The proof of Theorem~\ref{thm:poincare} is really a corollary of the
proof of Theorems~\ref{thm:mainint} and~\ref{thm:sparse}.  First, we will show that it will
suffice to assume $|E|<\infty$ and prove \eqref{eqn:poincare2} with the left-hand side
replaced by
\[ \frac{1}{u(E)}\int_E |f(x)-f_E|^q u(x)\,dx, \]
where $f_E = \avgint_E f(x)\,dx$.  For by H\"older's inequality,
\begin{align*}
& \int_E |f(x)-f_{E,u}|^q u(x)\,dx \\
& \qquad \qquad \lesssim \int_E |f(x)-f_E|^q u(x)\,dx + \int_E |f_E-f_{E,u}|
  u(x)\,dx \\
& \qquad \qquad  = \int_E |f(x)-f_E|^q u(x)\,dx
+ u(E) \left| \frac{1}{u(E)}\int_E (f(x) - f_E) u(x)\,dx \right|^q \\
& \qquad \qquad \leq 2 \int_E |f(x)-f_E|^q u(x)\,dx.
\end{align*}
To get \eqref{eqn:poincare2} with $E$ such that $|E|=\infty$, replace
$E$ by $E'= E\cap B_R(0)$.  Then $E'$ is convex and
$v(E'),\,|E'|<\infty$.  The desired inequality follows from Fatou's
lemma if we let $R\rightarrow \infty$.

Next, recall that for convex sets $E$, we have the following well-known
inequality (see~\cite{MR1814364}):  for scalar functions $f\in C^1(E)$ and $x\in E$,
\[ |f(x)-f_E| \lesssim \int_E \frac{|\nabla f(y)|}{|x-y|^{d-1}}\,dy
  = I_1(\chi_E |\nabla f|)(x). \]
Therefore, it will be enough
to prove that given any vector-valued function $g$,
\begin{equation} \label{eqn:mixed}
 \|u^{\frac{1}{q}} I_1(|V^{-\frac{1}{p}}g|)\|_{L^q} \lesssim
    \|g\|_{L^p}.
  \end{equation}
For in this case, if we let let $g=\chi_E V^{\frac{1}{p}}\nabla f$, then
combining the above inequalities we get inequality~\eqref{eqn:poincare2}.

\medskip

To prove \eqref{eqn:mixed} we argue as in the proof of
Theorems~\ref{thm:mainint} and~\ref{thm:sparse}, so here we only sketch the main ideas.
Define the matrix $U$ to be the diagonal
matrix $u(x)I_d$, where $I_d$ is the $d\times d$ identity matrix.  Let
$\Uu_Q^\Psi$ and $\Vv_Q^\Phi$ be the reducing operators associated to
$U$ and $V$ as in Section~\ref{section:proof-mainint}.   Fix a vector
function $g$ and a scalar function $h \in L^{q'}$; without loss of generality we may
assume $g$ and $h$ are bounded functions of compact support.  By the scalar
theory of domination by sparse operators for the fractional integral
(see~\cite{CruzUribe:2016ji}), we have that
\begin{equation*}
\left|\ip{u ^\frac{1}{q} I_\alpha(| V^{-\frac1p} {g}|)}{{h}}_{L^2} \right|
\lesssim \sum_{t \in \{0, \pm\frac13\}^d} \sum_{Q \in \Ss^t} {|Q|
  ^{\frac{\alpha}{d}}} \avgint_Q \int_Q | V(y)^{-\frac{1}{p}}
  {g}(y)|
|u (x) ^\frac{1}{q} {h}(x)| \, dx \, dy,
\end{equation*}
where each $\Ss^t$ is a sparse set contained in the dyadic grid $\D^t$
which is defined as in
Proposition~\ref{dyadic}.
Therefore, we need to fix a sparse set
$\Ss$ and show that the inner sum is bounded by
$\|g\|_{L^p} \|h\|_{L^{q'}}$.

Let $\{e_j\}$ be any orthonormal basis
of $\C^n$.  Then
\begin{align*}
& \sum_{Q \in \D^t} {|Q|
  ^{\frac{\alpha}{d}}} \avgint_Q \int_Q | V(y)^{-\frac{1}{p}}
  {g}(y)|
|u (x) ^\frac{1}{q} {h}(x)| \, dx \, dy \\
& \qquad \quad \leq
\sum_{Q \in \Ss} {|Q|
  ^{\frac{\alpha}{d}}} \avgint_Q \int_Q | (\Vv_Q^\Phi)^{-1} V(y)^{-\frac{1}{p}}
  {g}(y)|
|\Vv_Q^\Phi|_\op |u (x) ^\frac{1}{q} {h}(x)| \, dx \, dy \\
& \qquad \quad \lesssim
\sum_{Q \in \Ss} \sum_{j=1}^n {|Q|
  ^{1+\frac{\alpha}{d}}} \avgint_Q | (\Vv_Q^\Phi)^{-1} V(y)^{-\frac{1}{p}}
  {g}(y)|\,dy
\avgint_Q  |\Vv_Q^\Phi U (x) ^\frac{1}{q} {h}(x)e_j| \, dx \\
& \qquad \quad  \leq \sup_{{Q} }
  |{Q}|^{\frac{\alpha}{d}+\frac{1}{q}-\frac{1}{p}}
|\MC{V}_{{Q}} ^\Phi \MC{U}_{{Q}} ^\Psi|_{\op} \\
& \qquad \quad \quad \times
\sum_{j=1}^n \sum_{Q \in \Ss} |E_Q|  \,|Q| ^ \frac{\beta}{d} \dashint_Q
  |(\MC{V}_Q ^\Phi) ^{-1} V^{-\frac{1}{p}}(y) f(y)|\,dy \,
  \dashint_Q |(\MC{U}_Q ^\Psi)^{-1} U^\frac{1}{q}(x) h(x)e_j| \,dx.
\end{align*}
The middle inequality holds since $u$ and $h$ are
scalars and $|\Vv_Q^\Phi|_\op \approx \sum |\Vv_Q^\Phi e_j|$.

The proof now continues exactly as before. To estimate the supremum in
the last inequality, note that by~\eqref{eqn:reducing-norm2} it is
equivalent to~\eqref{eqn:poincare1} which is finite by assumption.

\bigskip

Finally, we use Theorem~\ref{thm:poincare} to prove the existence of a
weak solution of a degenerate $p$-Laplacian equation.  In a recent
paper~\cite{CruzUribe:2017wz} it was shown that the existence of a
weak solution was equivalent to the existence of a $(p,p)$ Poincar\'e
inequality.   For brevity, we refer the reader
to~\cite{CruzUribe:2017wz} for precise definitions of a weak solution,
which is technical in the degenerate case.

\begin{corollary} \label{cor:p-laplacian}
Fix $1<p<\infty$ and a bounded,  convex, open set $E \subset \R^d$.
Let $u$ be a scalar weight and $A$  a matrix weight such that $|A|_\op^{\frac{p}{2}}
\in L^1_\loc(E)$.  Suppose that there exist Young functions $\Phi$ and
$\Psi$, $\bar{\Phi}\in B_p$ and $\bar{\Psi}\in B_{p'}$, such that
\begin{equation} \label{eqn:p-laplacian1}
 \sup_Q |Q|^{\frac{1}{d}}\|u^{\frac{1}{p}}\|_{\Psi,Q}
  \|A^{-\frac{1}{2}}\|_{\Phi,Q} < \infty.
\end{equation}
Then for every $f\in L^p(u; E)$ there exists a weak solution $g$ to the
degenerate $p$-Laplacian Neumann problem
\begin{equation} \label{eqn:p-laplacian2}
\begin{cases}
\dv\Big(\Big|\sqrt{A(x)}\nabla g(x)\Big|^{p-2} A(x)\nabla g(x)\Big)
& = |f(x)|^{p-2}f(x)u(x) \text{ in }E\\
{\bf n}^t \cdot A(x) \nabla u &= 0\text{ on }\partial E,
\end{cases}
\end{equation}
where ${\bf n}$ is the outward unit normal vector of
$\partial E$.
\end{corollary}

\begin{remark}
In the statement of Corollary~\ref{cor:p-laplacian} there seems to be
an implicit assumption on the regularity of $\partial E$ so that
${\bf n}$ exists.  This is not the case, but we refer the reader
to~\cite{CruzUribe:2017wz} for details.
\end{remark}

\begin{proof}
Define the matrix weight $V$ by $A^{\frac{1}{2}}=V^{\frac{1}{p}}$.
Then \eqref{eqn:p-laplacian1} is equivalent to~\eqref{eqn:poincare1}.
Therefore, by Theorem~\ref{thm:poincare} we have the
Poincar\'e inequality
\[ \int_E |f(x)-f_{E,u}|^pu(x)\,dx
\lesssim \int_E |A^{\frac{1}{2}}(x) \nabla f(x)|^p\,dx. \]
But by the main result in~\cite{CruzUribe:2017wz}, this is equivalent
to the existence of a weak
solution to \eqref{eqn:p-laplacian2}.
\end{proof}


\bibliographystyle{plain}

\bibliography{Matrixbumps}

\end{document}